\documentclass{amsart}

\usepackage[utf8]{inputenc}
\usepackage[T1]{fontenc}
\usepackage{lmodern}
\usepackage[leqno]{amsmath}
\allowdisplaybreaks
\usepackage{amssymb} 
\usepackage{amsthm}
\usepackage{mathtools}
\usepackage[dvipsnames,svgnames,table]{xcolor}
\usepackage{xspace}
\usepackage{embrac} 
\ChangeEmph{(}[-.01em,.04em]{)}[.04em,-.05em]
\usepackage{microtype}
\usepackage{textcase}
\usepackage{cancel}

\usepackage{enumitem}
\usepackage[acronym]{glossaries}
\usepackage{graphicx}
\usepackage{float}
\usepackage{booktabs}
\usepackage{csquotes}
\usepackage{multirow}
\usepackage{tikz}
\usepackage{caption}
\usepackage{subcaption}
\usetikzlibrary{angles,quotes,calc,positioning,shapes,matrix,arrows}
\usepackage{tikz-cd}
\usetikzlibrary{}
\usepackage{tkz-euclide}
\usepackage{etoolbox}
\usepackage{url}
\usepackage[hyperindex=true, colorlinks=true, linkcolor=Maroon, citecolor=DarkGreen, urlcolor=NavyBlue, breaklinks=true,linktocpage=true,linktoc=all]{hyperref}
\usepackage[english,capitalise,noabbrev]{cleveref}
\crefname{equation}{equation}{equations}
\usepackage[logical-quotes,backrefs,initials]{amsrefs}
\usepackage{lipsum}

\renewcommand{\S}{\mathbb{S}}
\newcommand{\Esp}{\mathbb{E}}
\newcommand{\Energy}{E}

\renewcommand{\digamma}{\psi}
\newcommand{\N}{\mathbb{N}}

\newcommand{\R}{\mathbb{R}}
\newcommand{\C}{\mathbb{C}}

\renewcommand{\S}{\mathbb{S}}

\newcommand{\CP}{\mathbb{CP}}

\renewcommand{\Re}{\operatorname{Re}} 
\renewcommand{\Im}{\operatorname{Im}} 

\newcommand{\Jac}{\operatorname{Jac}}

\DeclarePairedDelimiterX\norm[1]\lVert\rVert{\ifblank{#1}{\:\cdot\:}{#1}}
\DeclarePairedDelimiterX\abs[1]\lvert\rvert{\ifblank{#1}{\:\cdot\:}{#1}}
\DeclarePairedDelimiterX\set[1]{\{}{\}}{\ifblank{#1}{\: \:}{#1}}
\DeclarePairedDelimiterX\innerprod[2]\langle\rangle
{\ifblank{#1#2}{\,\cdot,\cdot\,}
	{#1,#2}}
 \DeclarePairedDelimiterX\mutualEnergy[2]\langle\rangle{\ifblank{#1#2}{\,\cdot,\cdot\,}{#1,#2}}
\DeclarePairedDelimiterX\braket[2]\langle\rangle
{\ifblank{#1#2}{\,\cdot\,\vert\,\cdot\,}
	{\,#1\, \vert\, #2\,}}

\DeclarePairedDelimiterX\floor[1]\lfloor\rfloor{\ifblank{#1}{\:\cdot\:}{#1}}
\DeclarePairedDelimiterXPP\expectation[2]{\ifblank{#1}{\mathbb{E}}{\mathbb{E}_{#1}}}\lbrack\rbrack{}{\ifblank{#2}{\:\cdot\:}{#2}}
\DeclarePairedDelimiterXPP\Pochhammer[2]{}{(}{)}{_{#2}}{\ifblank{#1}{\cdot}{#1}}

\newcommand{\Sphericalensemble}[1]{\mathcal{X}_S^{#1}}
\newcommand{\Harmonicensemble}[1]{\mathcal{X}_H^{#1}}
\newcommand{\HopfSphericalensemble}[2]{\uparrow_{H}^{#1}(\mathcal{X}_S^{#2})}
\newcommand{\HopfHarmonicensemble}[2]{\uparrow_{H}^{#1}(\mathcal{X}_H^{#2})}
\newcommand{\Hopfdpp}[2]{\uparrow_{H}^{#1}(\mathcal{X}^{#2})}
\newcommand{\HopfDiamond}[2]{\uparrow_{H}^{#1}(\diamond_{#2})}
\newcommand{\HopfUniform}[2]{\uparrow_{H}^{#1}(U_{\S^2}^{#2})}
\newcommand{\HopfUniformAntipodal}[2]{\uparrow_{H}^{#1}(\bar{U}_{\S^2}^{#2})}

\newacronym{dpp}{\textup{\textsc{{dpp}}}}{\emph{determinantal point process}}

\theoremstyle{plain}
\newtheorem{theorem}{Theorem}[section]

\newtheorem{corollary}[theorem]{Corollary}

\newtheorem{lemma}[theorem]{Lemma}
\newtheorem{proposition}[theorem]{Proposition}

\theoremstyle{definition}

\newtheorem{definition}[theorem]{Definition}

\theoremstyle{remark}

\newtheorem{remark}[theorem]{Remark}

\begin{document}

	\title[Random points on $\S^3$ with small logarithmic energy]{Random points on $\S^3$ with small logarithmic energy}
	
	\author{Ujué Etayo}
	\address{Ujué Etayo: Departamento de Matemáticas, CUNEF Universidad, Leonardo Prieto Castro, 2. Ciudad Universitaria, 28040 Madrid, Spain}
	\email{ujue.etayo@cunef.edu}
	
	\author{Pablo G. Arce}
	\address{Pablo G. Arce:  Institute of Mathematical Sciences, National Spanish Research Council, C. Nicolás Cabrera, 13-15, Fuencarral-El Pardo, 28049, Madrid, Spain}
	\email{pablo.garcia@icmat.es}
	
	\date{\today{}}
	
	\thanks{The authors have been supported by grant PID2020-113887GB-I00 funded by MCIN/AEI/10.13039/501100011033. The first author has also been supported by the starting grant from FBBVA associated with the prize José Luis Rubio de Francia and by the Ramón y Cajal Programme of the Spanish Ministry of Science, Innovation and Universities, through the Agencia Estatal de Investigación (AEI), co-funded by the European Social Fund Plus (ESF+), under grant RYC2024-049105-I}
	
	

\begin{abstract}
We analyse several constructions of random point sets on the sphere $\S^{3}\subset\R^4$ evaluating and comparing them through their discrete logarithmic energy:
%
\begin{equation*}
    E_0(\omega_N) = 
    \sum_{\substack{i, j=1\\ i \neq j}}^{N} 
    \log\frac{1}{\|x_i - x_j\|}, 
    \; \text{ where}\; \omega_N=\{x_1,\ldots,x_N\} \subset \S^3.
\end{equation*}
%
Using the Hopf fibration, we lift a range of well-distributed families of points from the $2$-dimensional sphere - including uniformly random points, antipodally symmetric sets, determinantal point processes, and the Diamond ensemble - to $\S^{3}$, in order to assess their energy performance.
In particular, we carry out this asymptotic analysis for the Spherical ensemble (a well known determinantal point process on $\S^2$), obtaining as a result a family of points on the $3$-dimensional sphere whose logarithmic energy is asymptotically the lowest achieved to date. 
This, in turn, provides a new upper bound for the minimal logarithmic energy on $\S^3$.
Although an analytic treatment of the lifted Diamond ensemble remains elusive, extensive simulations presented here show that its empirical energies lie below all other deterministic and non-deterministic constructions considered.  
Together, these results sharpen the quantitative link between potential-theoretic optima on \(\S^{2}\) and \(\S^{3}\) and provide both theoretical and numerical benchmarks for future work.
\end{abstract}
	
	\maketitle



\section{Introduction and main results}

\subsection{Logarithmic energy on $\S^2$}

The problem of distributing points on a sphere so as to minimize a certain energy or maximize mutual distances is a classic topic in potential theory and discrete geometry. One notable example is \textit{Whyte's problem}, which asks for the arrangement of $N$ points on the unit $2$-dimensional sphere, $\mathbb{S}^2$, that maximizes the product of all pairwise distances. 
Equivalently, this problem seeks to minimize the \textit{logarithmic energy} of $N$ points on $\mathbb{S}^2$. For a configuration $\omega_N = \{x_1,\dots,x_N\}\subset \mathbb{S}^2$, the discrete logarithmic energy is defined by 
\begin{equation*}
    E_0(\omega_N) := 
    \sum_{\substack{i, j=1\\ i \neq j}}^{N} 
    \log\frac{1}{\|x_i - x_j\|}, 
\end{equation*}
where $\|\cdot\|$ denotes the Euclidean distance. 
Points which minimize $E_0(\omega_N)$ are often called \textit{logarithmic points} or \textit{elliptic Fekete points}.
The problem of minimizing the logarithmic energy has special significance - for instance, it features in Smale’s seventh problem on the asymptotic distribution of optimal configurations on the sphere, see \cite{Smale1998}.

Potential theory shows that the equilibrium measure for the logarithmic potential on $\mathbb{S}^d$ is the normalized surface area measure \cite{BorodachovHardinSaff2019}. Accordingly, one expects that, as $N\to\infty$, configurations minimizing (or nearly minimizing) the logarithmic energy $E_0$ become asymptotically uniformly distributed on the sphere.

In the case of $\mathbb{S}^2$, this equidistribution has been rigorously established, and the minimal logarithmic energy grows like $N^2$, in agreement with the energy of a uniformly charged spherical shell. Recent work has further refined this picture by deriving precise asymptotic expansions of the minimum energy; see \cite{Marzo2025} for a comprehensive account of the current state of the art.

\subsection{Minimal logarithmic energy on $\S^3$}
Compared to the $2$-dimensional sphere, much less is known about logarithmic energy minimization on the $3$-dimensional sphere
\begin{equation*}
\mathbb{S}^3=\{(x_1,x_2,x_3,x_4)\in\mathbb{R}^4: x_1^2+x_2^2+x_3^2+x_4^2=1\}.
\end{equation*}
Finding and characterizing optimal discrete configurations on $\mathbb{S}^3$ is a challenging problem. 
The higher dimensional geometry and the absence of a simple explicit parametrization for $\mathbb{S}^3$ make the problem non-trivial. 
The current state of the art of the asymptotic expansion of the minimal logarithmic energy on $\S^3$ is due to different authors \cites{RakhmanovSaffZhou1994,brauchart_next-order_2012,BorodachovHardinSaff2019,Brauchart_optimal_logarithmic} and is presented in the following equation.
\begin{equation}\label{expansion}
    \min_{\omega_N \subset \S^3} \Energy_{0}(\omega_N)
    =
    - \frac{N^2}{4}
    - \frac{1}{3}N\log N
    + \mathcal{O}(N),     
    \; \text{ with}\; \omega_N=\{x_1,\ldots,x_N\} \subset \S^3.
\end{equation}
The best upper bound for the minimal logarithmic energy known to date is given in \cite{beltran_energy_2016}, where the authors introduce a random configuration of $N$ points called the Harmonic ensemble, denoted in the present work by $\Harmonicensemble{N}$. 
In \cite{beltran_energy_2016}, the authors compute the asymptotic expansion of the expected logarithmic energy of $N$ points drawn from the Harmonic ensemble when $N$ goes to infinity, obtaining
\begin{equation}\label{eq:previousbound}
\Esp_{\omega_N \sim \Harmonicensemble{N}}
\left[
E_{0}(\omega_N)
\right]
=
-\frac{1}{4}N^2
-\frac{1}{3}N\log N
+C_H N
+o(N),   
\end{equation}
with
\begin{equation*}
C_H=\frac{1}{3}\log \frac{1}{3}+\log 2+\psi_0\left(\frac{3}{2}\right)+\frac{1}{3}
\approx
0.70\ldots
\end{equation*}
where $\psi_0=(\log\Gamma)'$ is the digamma function.
This expansion provides the sharpest known upper bound for the minimal logarithmic energy.

\subsection{Hopf fibration}

In this work, we introduce and advocate for the use of the \emph{Hopf fibration} in tackling the $\S^3$ problem. 
The Hopf fibration is a remarkable construction introduced by H. Hopf in 1931 \cite{Hopf1931}, which realizes the $3$-dimensional
sphere as a non-trivial circle bundle over a $2$-dimensional base.

A natural and intrinsic formulation of the Hopf fibration is obtained by viewing $\S^3$ as the unit sphere in $\C^2$,
\begin{equation*}
    \S^3=\{(z_1,z_2)\in\C^2:\ |z_1|^2+|z_2|^2=1\},
\end{equation*}
and by considering the free action of the circle group
$\S^1=\{e^{i\theta}:\theta\in[0,2\pi)\}$ given by
\begin{equation*}
    e^{i\theta}\cdot(z_1,z_2)=(e^{i\theta}z_1,e^{i\theta}z_2).
\end{equation*}
The orbits of this action are circles contained in $\S^3$, and the corresponding
quotient space is the complex projective line,
\begin{equation*}
    \S^3/\S^1 \;\cong\; \CP^1.
\end{equation*}
The associated quotient map
\begin{equation}\label{inversaHopfcompleja}
    \pi:\S^3\longrightarrow\CP^1,\qquad (z_1,z_2)\longmapsto [z_1:z_2],
\end{equation}
defines the Hopf fibration with base space $\CP^1$ and fibre $\S^1$.

The classical formulation of the Hopf fibration takes $\S^2$ as its base, which is recovered by identifying $\CP^1$ with the Riemann sphere.

\begin{proposition}\label{prop:diffCP1S2}
There exists a $C^{1}$ (in fact, smooth) diffeomorphism
\begin{equation*}
\Phi:\S^{2}\longrightarrow \CP^{1}
\end{equation*}
such that the push-forward of the spherical surface measure on $\S^{2}$ coincides, up to a constant factor, with the Fubini-Study volume form on $\CP^{1}$. More precisely, if $\sigma$ denotes the surface area measure on $\S^{2}$ and $d\mathrm{vol}_{FS}$ the Fubini--Study volume form on $\CP^{1}$ normalized by ${\mathrm{vol}_{FS}}(\CP^{1})=\pi$, then
\begin{equation*}
d\mathrm{vol}_{FS}=\frac{1}{4}\,\Phi_{\#}\sigma.
\end{equation*}
In particular, the Jacobian determinant of $\Phi$ satisfies
\begin{equation*}
\Jac(\Phi)=\frac{1}{4},
\qquad
\Jac(\Phi^{-1})=4.
\end{equation*}
\end{proposition}

\begin{proof}[Proof (sketch)]
Identify $\CP^{1}$ with the Riemann sphere and consider the stereographic projection from the north pole $N=(0,0,1)\in\S^{2}$,
\begin{equation*}
\phi:\S^{2}\setminus\{N\}\longrightarrow\C,
\qquad
\phi(x,y,z)=\frac{x+iy}{1-z}.
\end{equation*}
Extending by $\phi(N)=\infty$ and composing with the standard affine chart $\C\hookrightarrow\CP^{1}$ yields a smooth bijection
\begin{equation*}
\Phi:\S^{2}\longrightarrow\CP^{1}.
\end{equation*}
Its inverse is explicitly given, in affine coordinates $w\in\C$, by
\begin{equation*}
\Phi^{-1}([1:w])
=
\left(
\frac{2\Re w}{1+|w|^{2}},
\frac{2\Im w}{1+|w|^{2}},
\frac{|w|^{2}-1}{1+|w|^{2}}
\right),
\end{equation*}
which shows that $\Phi$ is a $C^{1}$ diffeomorphism (in fact, $C^{\infty}$).

In stereographic coordinates, the spherical surface measure and the Fubini-Study volume form are given by
\begin{equation*}
d\sigma=\frac{4}{(1+|w|^{2})^{2}}\,dA(w),
\qquad
d\mathrm{vol}_{FS}=\frac{1}{(1+|w|^{2})^{2}}\,dA(w),
\end{equation*}
where $dA$ denotes the Lebesgue measure on $\C$. Comparing these expressions yields
\begin{equation*}
d\mathrm{vol}_{FS}=\frac{1}{4}\,d\sigma,
\end{equation*}
which implies the stated values of the Jacobian determinants.
This expression is consistent with the earlier identity \(d\mathrm{vol}_{\mathrm{FS}} = \frac{1}{4}\, \Phi_{\#} \sigma\), as we are here working in local coordinates on \(\mathbb{C} \subset \CP^1\), where the pushforward \(\Phi_{\#} \sigma\) coincides with the standard surface area measure \(d\sigma\) on the Riemann sphere.

\end{proof}

Composing the projection $\pi$ with the diffeomorphism
\begin{equation*}
    \Phi^{-1}:\CP^1\longrightarrow\S^2    
\end{equation*}
introduced in \cref{prop:diffCP1S2}, one obtains the map
\begin{equation*}
    h=\Phi^{-1}\circ\pi:\S^3\longrightarrow\S^2,
\end{equation*}
which admits the explicit real-coordinate expression
\begin{equation*}
    h(a,b,c,d)=(a^2+b^2-c^2-d^2,\;2(ad+bc),\;2(bd-ac)),
\end{equation*}
for $(a,b,c,d)\in\S^3\subset\R^4$.

It is easy to check that the image of $h$ lies on $\S^2$, since the sum of the squares of its components equals $(a^2+b^2+c^2+d^2)^2=1$. 
For each point
$\mathbf P\in\S^2$, the preimage $h^{-1}(\mathbf P)$ is a great circle in $\S^3$, called the fibre over $\mathbf P$. 
Thus, $\S^3$ is decomposed into disjoint circles, one above each point of the base space.

An explicit parametrization of the fibre over a point $(p_1,p_2,p_3)\in\S^2$ is
given by
\begin{multline}\label{inversa_hopf}
h^{-1}(p_1,p_2,p_3)
\\
=
\frac{\Bigl(
(1+p_1)\cos t,\,
(1+p_1)\sin t,\,
p_2\sin t-p_3\cos t,\,
p_2\cos t+p_3\sin t
\Bigr)}{\sqrt{2(1+p_1)}},
\end{multline}
with $t\in[0,2\pi)$.

This viewpoint shows that the Hopf fibration may be equivalently regarded as a fibration over $\S^2$ or over $\CP^1$, the two base spaces being related by the diffeomorphism $\Phi$. 
In the present work, this dual interpretation will play a key role in transferring point processes and geometric structures between $\S^3$, $\S^2$, and $\CP^1$.


\subsection{Main results}

\subsubsection{Fibred constructions}

Via the Hopf fibration, well-distributed configurations on $\mathbb{S}^2$ can be lifted to $\mathbb{S}^3$, yielding a direct link between point sets on the two spheres.
For example, suppose that we start with a configuration of $r$ "base" points on $\mathbb{S}^2$ that are reasonably uniformly distributed. 
Above each such base point $y_{j}\in \S^2$, consider its fibre $h^{-1}(y_{j})\cong \S^1$ in $\mathbb{S}^3$. 
We can place $k$ points along each of these fibres (for instance, $k$ equally spaced points on the circle). 
This construction produces a set of $N=r\,k$ points on $\mathbb{S}^3$. 
If the $r$ base points on $\mathbb{S}^2$ are well-spaced and the $k$ fibre points on each circle are evenly spaced, then one expects that the resulting $N$-point configuration on $\mathbb{S}^3$ inherits a good separation and distribution. 
Recent results support this approach; see, for instance, \cites{BeltranEtayo2018}, where the Hopf fibration is generalized to higher-dimensional spheres and exploited as a tool for constructing well-distributed point sets.  
A similar strategy is adopted by Beltrán, Carrasco, Ferizović and López-Gómez \cite{BeltranCarrascoFerizovicLopezGomez2025} to distribute points in $SO(3)$.

In the rest of this article, we develop the theoretical framework for this lifting technique and analyse the logarithmic energy of the resulting configurations. 
We work with points sampled uniformly from $\S^2$, together with their antipodal counterparts.
For these point configurations, we obtain the asymptotic
\begin{equation*}
    \Esp[\Energy_{0}(\omega_N)]
    =
    \frac{-1}{4}N^2+CN,  
    \; \text{ where}\; \omega_N=\{x_1,\ldots,x_N\} \subset \S^3,
\end{equation*}
with $C$ a constant depending on the family of points.
One can check that the first term of the asymptotic expansion is correct, but the second term is already of the wrong order (see \cref{expansion}).
A comparison of the linear-term constant $C$ for the different families of uniform points is given in \cref{FiguraUniforme}.
\begin{table}[H]
    \centering
{\renewcommand{\arraystretch}{1.3}
\begin{tabular}{ll}
    Distribution & Constant on the linear term \\ 
    \hline \hline
    Uniform $\S^3$ & $1/4$\\
    Uniform $\S^3$ and symmetry &  $\frac{1}{2}-\log2=-0.193...$\\
    Uniform $\S^2$ and fibres &  $1-2\log2=-0.386...$\\
    Uniform $\S^2$, symmetry and fibres \qquad  &  $\frac{7}{2}(1-\log2)-\log7=-0.871...$
\end{tabular}}
    \caption{Linear-term coefficients for different uniform distributions.}
    \label{FiguraUniforme}
\end{table}
We also study other types of random process, known as determinantal point processes.
In particular, we consider the Harmonic ensemble on the $2$-dimensional sphere introduced in \cite{beltran_energy_2016}, and the Spherical ensemble introduced in \cite{krishnapur_random_2006}.
For the latter, we obtain the main theorem of the paper.
Finally, we study point sets coming from the Diamond ensemble, as defined in \cite{BeltranEtayo2020}.

\subsubsection{Notation}

We refer to a sequence of sets of points $\{\omega_r\}_{r\in\N}$ where each $\omega_r = \{x_1,\ldots,x_r\}$ is a set of $r$ points on the corresponding sphere as a \emph{family of points}.
In particular, we work with the following families:
\begin{itemize}
    \item $U_{\S^2}^{r}$ is a  set of $r$ uniformly distributed points on $\S^2$.
    \item $\bar{U}_{\S^2}^{r}$ is a  set consisting of $r/2$ uniformly distributed points on $\S^2$ and their $r/2$ antipodal counterparts.
    \item $\Harmonicensemble{r}$ is a set of $r$ points on $\S^2$ following the distribution of the Harmonic ensemble.
    \item $ \Sphericalensemble{r}{}$ is a set of $r$ points on $\S^2$ following the distribution of the Spherical ensemble.
    \item $\diamond(r)$ is a set of $r$ points on $\S^2$ following the distribution of the Diamond ensemble.
\end{itemize}
We use the symbol $\uparrow_{H}^{k}(\omega_r)$ to denote the set  of $N=kr$ points obtained by the procedure of considering $k$ roots of unity on the fibre corresponding to each point in $\omega_r$. 
For example, if we fibre $r$ points coming from the Spherical ensemble in $\S^2$ and take $k$ points in each fibre, we denote it by $\HopfSphericalensemble{k}{r}$.

\subsubsection{A new bound for the minimal logarithmic energy on $\S^3$}

The \emph{Spherical ensemble} is a determinantal point process on the Riemann sphere that arises in random matrix theory.  
It was first obtained by stereographically projecting the generalized eigenvalues of two independent Ginibre matrices onto the sphere \cite{krishnapur_random_2009}.  
This construction yields a rotationally invariant distribution of $N$ points on $\S^2$.  
In particular, the Spherical ensemble has a joint density proportional to the product of squared pairwise distances (a spherical Vandermonde determinant), which is the hallmark of determinantal point processes and leads to strong mutual repulsion between points.  
Consequently, the points tend to spread out uniformly on the sphere.
Despite being random, this ensemble exhibits nearly optimal uniformity in the sense of potential theory: it achieves very low discrepancy and near-minimal Riesz energy, on par with (and in some cases outperforming) the best deterministic point configurations \cites{Kuijlaars1998,Alishahi2015,brauchart_next-order_2012}.  

In this paper, we prove that the expected logarithmic energy of points coming from the fibred Spherical ensemble provides the best upper bound for the minimal logarithmic energy on $\mathbb{S}^3.$

\begin{theorem}\label{thm:main}
Let $\frac{A}{B}$ be a rational approximation of $\frac{3^2\pi}{4^3}\approx 0.44\ldots$. 
Let $\HopfSphericalensemble{k}{r}$ be a set of $N = kr$ points obtained by taking $k$ points  equally spaced, randomly rotated in each fibre $\S^1$ of $\S^3$, with respect to $r = \frac{A}{B}k^2$ points drawn from the Spherical ensemble $\Sphericalensemble{r}$ in $\mathbb{S}^2$.
Then,
\begin{multline*}
\Esp_{\omega_N \sim\HopfSphericalensemble{k}{r}} 
\left[
\Energy_{0}(\omega_N)
\right]
\\
=
-\frac{N^2}{4}
-\frac{1}{3}N\log N
+ N\left(
\frac{\sqrt{\pi}}{4}
\left(\frac{B}{A}\right)^{1/2}
-
\log\left(\left(\frac{B}{A}\right)^{1/3}\right)
\right)
+
o(N)
.    
\end{multline*}
In particular, the coefficient of the linear term tends to $\frac{1}{3}\left(2+\log\frac{9\pi}{64}\right)$ as $\frac{A}{B}$ more closely approximates $\frac{3^2\pi}{4^3}$.
\end{theorem}

\begin{remark}
Note that for a given number of points $k$ we have to find a suitable approximation $\frac{A}{B}$ that makes $r$ an integer number.
\end{remark}

This leads us to the following upper bound for the minimal logarithmic energy on $\S^3$.

\begin{corollary}\label{main:cor}
For an infinite sequence of positive integer numbers $\{n_i\}_{i\in\N}$, the following upper bound holds:
\begin{equation*}
\min_{\omega_{n_i} \subset \S^3} \Energy_{0}(\omega_{n_i})
\leq
-\frac{N^2}{4}
-\frac{1}{3}N\log N
+C_S N
+ o(N),
\end{equation*}
with
\begin{equation*}
C_S
=
\frac{1}{3}\left(2+\log\frac{9\pi}{64}\right).
\end{equation*}
\end{corollary}

We recall that, to date, the best result is the one obtained by the authors in \cite{beltran_energy_2016}, presented in \cref{eq:previousbound}.
The asymptotic expansions given by both families match up to the linear term, as both correctly reproduce the terms in $N^2$ and $N\log N$ seen in  \cref{expansion}. 
Comparing the coefficients of the linear terms in both expansions, we have:
\begin{equation*}
C_S
=\frac{1}{3}\left(2+\log\frac{9\pi}{64}\right)
=0.39... 
<
C_H
=\frac{1}{3}\log \frac{1}{3}+\log 2+\psi_0\left(\frac{3}{2}\right)+\frac{1}{3}
=0.70...   
\end{equation*}
This allows us to conclude that the Spherical ensemble, in conjunction with the Hopf fibration, provides a point process on $\S^3$, $\HopfSphericalensemble{k}{r}$, with asymptotically lower logarithmic energy than the Harmonic ensemble, resulting in the configuration of points with the lowest logarithmic energy proved to date.

\subsubsection{The lifted Harmonic ensemble}

The Harmonic ensemble provides a natural and well-studied determinantal point process on $\mathbb{S}^2$, and its invariance properties make it a convenient benchmark for energy comparisons after lifting via the Hopf fibration. The point configurations obtained in this way reproduce both the leading and second-order terms in the asymptotic expansion of the logarithmic energy. However, as shown in \cref{cor:Harm:asymp}, the third-order term deviates from the expected behaviour: instead of being linear, it exhibits an additional $\log\log N$ factor. This deviation indicates a limitation of the lifting procedure when applied to the Harmonic ensemble, particularly in capturing the finer asymptotic structure.

\subsubsection{The lifted Diamond ensemble}

Among the different families of $N$ points that have been proposed on the $2$-dimensional sphere, the \emph{Diamond ensemble} introduced by Beltrán and Etayo \cite{BeltranEtayo2020} and then refined in \cite{BeltranEtayoLopezGomez2023} presently attains the lowest known logarithmic energy.   
The theoretical lower bound for the minimal logarithmic energy of $N$ points on $\S^{2}$ behaves asymptotically as
\begin{equation}\label{minlog2}
\Bigl(\tfrac12-\log 2\Bigr)N^{2}\;-\;\tfrac12 N\log N\;+\;C\,N,
\; C = -0.0568456\ldots, \; \text{see \cite{Marzo2025}.}
\end{equation}
Configurations drawn from the Diamond ensemble are $N\times0.0076\ldots$ far away from this rate and thereby outperform all other deterministic and random constructions currently available.

A configuration with $N$ points is built as follows.  
First, select $K$ parallels of latitude $z_k$ and allocate to the $k$th parallel $r_{k}$ points such that $\sum_{k=1}^{K} r_{k}=N-2$.
The $r_{k}$ points on the $k$th parallel are placed at equal angular separations in longitude, and the entire parallel is then rotated by a random phase chosen uniformly in $[0,2\pi)$.  
Finally, the North and South poles are appended, yielding a total of $N$ points.

Motivated by its success on $\S^{2}$, we attempted to transplant the Diamond ensemble to the $3$-dimensional sphere via the Hopf fibration.  
Despite the effort, the fibre-lifting technique did not lead to tractable analytic expressions for the resulting logarithmic energy, as shown in \cref{calculo_Diamond}.  
We therefore report here solely on the numerical evidence, which nevertheless offers meaningful insight into the performance of the lifted Diamond configurations on $\S^{3}$.
%
%
This has been carried out using \cite{noauthor_sympy_2017} from Python. 
The Numba library \cite{lam_numba_2015} has also been employed to accelerate computations by translating part of the code to C.

When transitioning to the sphere $\S^3$ using the Hopf fibration, one parameter whose optimal value is unknown is the number of points per fibre, denoted as $k$.
Various options have been experimented with for choosing $k$, and it has been found to be optimal to take $k=p^\alpha$, where $p$ is the number of parallels (a value that determines the other parameters of the Diamond ensemble). Multiple simulations have been conducted for different values of $\alpha$, with the optimal value being approximately $\alpha\approx 1.2$.

\begin{figure}[htbp]
    \centering
    \includegraphics[width=1\textwidth]{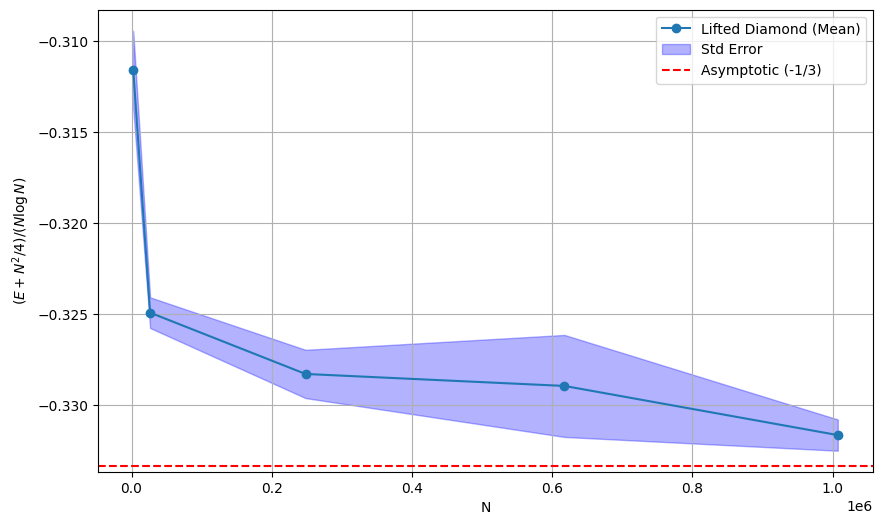}
    \caption{$\left(\Esp_{\omega_N\sim\HopfDiamond{k}{r}} \left[E_0(\omega_N)\right]+\frac{N^2}{4}\right)/N\log N$ plotted against the number of points, where $\Esp_{\omega_N\sim\HopfDiamond{k}{r}} \left[E_0(\omega_N)\right]$ represents the expected value of the logarithmic energy of the lifted configuration of points coming from the Diamond ensemble through the Hopf fibration. Averaged over 5 runs with different random seeds.}
    \label{fig:primer_orden}
\end{figure}

As shown in \cref{fig:primer_orden,fig:sec_orden}, for $\alpha = 1.2$ the coefficient of the $N\log N$ term in the asymptotic expansion of the energy tends to $-1/3$, in agreement with \cref{expansion}, while the coefficient of the linear term appears to approach $0$, which compares favourably with the other constructions considered in this work.

\begin{figure}[htbp]
    \centering
    \includegraphics[width=1\textwidth]{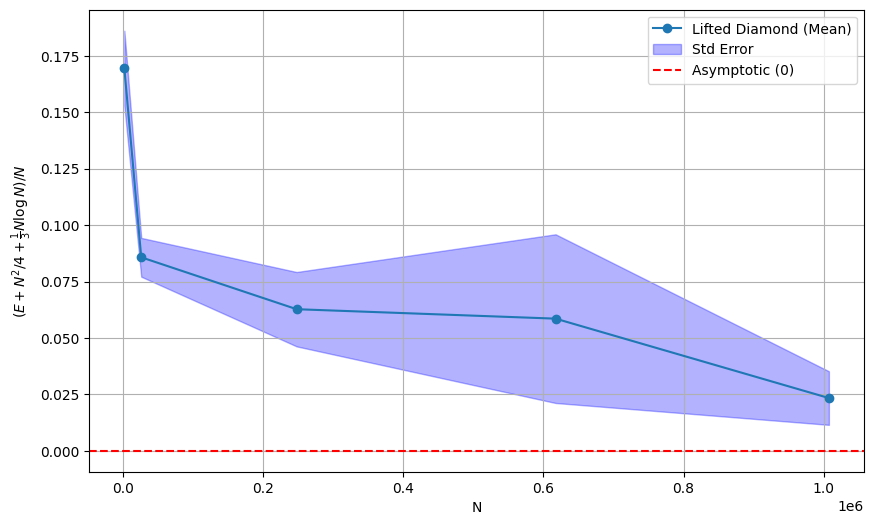}
    \caption{$\left(\Esp_{\omega_N\sim\HopfDiamond{k}{r}} \left[E_0(\omega_N)\right]+\frac{N^2}{4}+\frac{N\log N}{3}\right)/N$ plotted against the number of points, where $\Esp_{\omega_N\sim\HopfDiamond{k}{r}} \left[E_0(\omega_N)\right]$ represents the expected value of the lifted configuration of points coming from the Diamond ensemble through the Hopf fibration. Averaged over 5 runs with different random seeds.}
    \label{fig:sec_orden}
\end{figure}

\subsubsection{Organization of the paper}

\begin{description}
  \item[Section 1: Introduction and main results]  We state the logarithmic energy problem on \(\mathbb{S}^{3}\), introduce the Hopf fibration lifting strategy, and summarize the principal analytic and numerical findings.
  \item[Section 2: Uniformly distributed points] We analyse energies for configurations of points obtained from uniform distributions on \(\mathbb{S}^{3}\) and from uniform points on \(\mathbb{S}^{2}\) lifted by Hopf fibres.
  \item[Section 3: Determinantal Point Processes] We develop the potential-theoretic framework for determinantal point processes (Spherical ensemble, Harmonic ensemble, etc.) and derive their energy expansions after lifting.
  \item[Section 4: Proof of the main results] We collect auxiliary lemmas and give full proofs of the asymptotic bounds, including the new upper bound for the minimal logarithmic energy on $\S^3$.
  \item[Appendix A: The Diamond ensemble] We recall the construction of the Diamond ensemble on \(\mathbb{S}^{2}\), detail its fibre lifting to \(\mathbb{S}^{3}\), and report numerical results.
\end{description}

\section{Uniformly distributed points}\label{sec:unif}

We can sample i.i.d. uniformly distributed points in $\S^3$ by taking $(x,y,z,t)\subset\R^4$ such that
\begin{equation}\label{eq:unifdistS3}
\begin{cases}
x=\sqrt{1-v^2}\sqrt{1-u^2}\cos\phi \\
y=\sqrt{1-v^2}\sqrt{1-u^2}\sin\phi \\
z=\sqrt{1-v^2}u \\
t=v
\end{cases}
\text{ with }
\begin{cases}
\phi \in [0, 2\pi] \text{ uniform},\\
u \in [-1, 1] \text{ uniform},\\
v \in [-1, 1] \text{ with density } f_\mathbb{V}(v),\\
\end{cases}    
\end{equation}
where
\begin{equation*}
    f_\mathbb{V}(v)=\frac{2}{\pi}\sqrt{1-v^2}.
\end{equation*}

\begin{lemma}\label{lem:unif:S3}
Let $\omega_N=\{\mathbf{p}_1,\ldots,\mathbf{p}_N\}$ be a set of $N$ i.i.d. uniformly distributed points in $\S^3$, then:
\begin{equation*}
\Esp_{\substack{\phi_i, u_i, v_i, \\  1\leq i \leq N}}
[\Energy_{0}(\omega_N)]
=\frac{-1}{4}N^2+\frac{1}{4}N,    
\end{equation*}
where $\mathbf{p}_i = (x_i,y_i,z_i,t_i)$ are as in  \cref{eq:unifdistS3}.
\end{lemma}

\begin{proof}
Using the notation from \cref{eq:unifdistS3}, the distance between two random i.i.d. uniformly distributed points is
\begin{multline*}
||\mathbf{p}_i-\mathbf{p}_j|| 
\\
=\left[2-2\sqrt{1-v_i^2}\sqrt{1-v_j^2}\left(\sqrt{1-u_i^2} \sqrt{1-u_j^2}\cos(\phi_i-\phi_j)+u_iu_j\right)-2v_iv_j\right]^{1/2}.
\end{multline*}
$\S^3$ is a 2-point homogeneous space, that is, for every $\mathbf{p}_1, \mathbf{p}_2 ,\mathbf{q}_1, \mathbf{q}_2\in\S^3$ such that $\|\mathbf{p}_1- \mathbf{p}_2\| = \|\mathbf{q}_1- \mathbf{q}_2\|$ there exists an isometry of $\S^3$, that we denote by $i$, such that $i(\mathbf{p}_1)=\mathbf{q}_1$ and $i(\mathbf{p}_2)=\mathbf{q}_2$.
Hence we can simplify the computations of the expected value of the distance by taking $\mathbf{p}_i=(0,0,0,1)$. 
Then, for a fixed $j$,
\begin{multline}\label{eq:01}
    \Esp_{\phi_j, u_j, v_j}[-\log||(0,0,0,1)-\mathbf{p}_j||]
    \\
    =-\int_{-1}^1\frac{2}{\pi}\sqrt{1-v_j^2}\log(\sqrt{2-2v_j})dv_j
    \\
    =-\frac{1}{\pi}\int_{0}^{\pi}\sin^2 \theta\ln(2-2\cos\theta)d\theta
    \\
    =-\ln(2)
    - \frac{2}{\pi}\int_{0}^{\pi}\sin^2 \theta\ln \left(\sin\frac{\theta}{2}\right) d\theta
    ,
\end{multline}
after the appropriate change of variables $v_j = \cos\theta$.
For the integral, taking $\varphi=\frac{\theta}{2}$ and using that $\sin(2\varphi)=2\sin\varphi\cos\varphi$, we have
\begin{equation*}
    - \frac{2}{\pi}\int_{0}^{\pi}\sin^2 \theta\ln \left(\sin\frac{\theta}{2}\right) d\theta
    = 
    - \frac{2}{\pi}\int_{0}^{\pi/2}4\sin^2 \varphi\cos^2\varphi \ln \left(\sin\varphi\right) 2d\varphi.
\end{equation*}
The last integral resembles an incomplete beta function
\begin{equation*}
    I(a,b)
    =
    \int_{0}^{\pi/2} \sin^{2a}\varphi \cos^{2b}\varphi d\varphi
    =
    \frac{1}{2}B\left(a+\frac{1}{2},b+\frac{1}{2}\right).
\end{equation*}
If we differentiate the last expression with respect to $a$, we obtain
\begin{multline*}
    \frac{\partial I}{\partial a}(a,b)
    =
    \int_{0}^{\pi/2} 2\ln(\sin \varphi)\sin^{2a}\varphi \cos^{2b}\varphi d\varphi
    \\
    =
    \frac{1}{2}B\left(a+\frac{1}{2},b+\frac{1}{2}\right)
    \left[ \psi\left(a+\frac{1}{2}\right) + \psi\left(a+b+1\right)
    \right].    
\end{multline*}
It only remains to evaluate the last expression at $a=1$ and $b=1$:
\begin{equation*}
    \frac{\partial I}{\partial a}(1,1)
    =
    \frac{\pi(1-4\ln 2)}{32}.
\end{equation*}
We conclude with 
\begin{equation*}
    \Esp_{\phi_j, u_j, v_j}[-\log||(0,0,0,1)-\mathbf{p}_j||]
    =
    -\frac{1}{4}.
    \end{equation*}
The expected value of the logarithmic energy for $N$ points will then be
\begin{equation*}
\Esp_{\substack{\phi_i, u_i, v_i, \\  1\leq i \leq N}}
[\Energy_{0}(\omega_N)]
=
\frac{-1}{4}N(N-1) =\frac{-1}{4}N^2+\frac{1}{4}N,    
\end{equation*}
where $\omega_N=\{\mathbf{p}_1,\ldots,\mathbf{p}_N\}$ and $\mathbf{p}_i = (x_i,y_i,z_i,t_i)$ are as in  \cref{eq:unifdistS3}.
\end{proof}

We reason as follows: given a single point on the sphere, the point and the antipodal point (obtained through central symmetry on the sphere) form the best possible distribution of two points. Furthermore, for the energy calculation with a third point, each of the points can be treated as randomly and uniformly distributed points. Therefore, this suggests that generating $N/2$ points and then applying central symmetry should provide better-distributed points.

\begin{lemma}
Consider $M$ i.i.d.\ uniformly distributed points on $\S^3$, and let $\omega_N$ denote the set of $N = 2M$ points consisting of these $M$ points together with their antipodal points. 
Then:
\begin{equation*}
\Esp_{\substack{\phi_i, u_i, v_i, \\  1\leq i \leq M}}
[\Energy_{0}(\omega_N)]
    =
  \frac{-1}{4}N^2+\left(\frac{1}{2}-\log2\right)N,  
\end{equation*}
where $\omega_N=\{\mathbf{p}_1,\ldots,\mathbf{p}_N\}$ and $\mathbf{p}_i = (x_i,y_i,z_i,t_i)$ are as in  \cref{eq:unifdistS3} and $\mathbf{p}_{M+i}$ is the antipodal point of $\mathbf{p}_i$ for $1\leq i \leq M$.
\end{lemma}
\begin{proof}
We take $M=N/2$ i.i.d. uniformly distributed points in $\S^3$.
Then, the asymptotic expansion of the logarithmic energy consists on the sum of two terms:
\begin{enumerate}
    \item The energy contributed by antipodal points.
    \item The energy contributed by non-antipodal points.
\end{enumerate}
Fixing a point $\mathbf{p}_i$, we have from \cref{eq:01} that
\begin{equation*}
\Esp_{\phi_i, u_i, v_i,\phi_j, u_j, v_j}
[-\log||\mathbf{p}_i-\mathbf{p}_j||]=\frac{-1}{4},   \end{equation*}
for $\mathbf{p}_j$ drawn uniformly at random from $\S^3$. Furthermore, for the antipodal point, we have:
\begin{equation*}
\Esp_{\phi_i, u_i, v_i}
[-\log||\mathbf{p}_i-(-\mathbf{p}_i)||]
=
-\log||2\mathbf{p}_i||=-\log2.    
\end{equation*}
Taking both terms into account, the asymptotic expansion of the logarithmic energy is
\begin{equation*}
\Esp_{\substack{\phi_i, u_i, v_i, \\  1\leq i \leq M}}
[\Energy_{0}(\omega_N)]
=
\frac{-1}{4}N(N-2)-N\log2 =\frac{-1}{4}N^2+\left(\frac{1}{2}-\log2\right) N.
\end{equation*}
\end{proof}

One can verify that this energy is lower than that of the fully uniform distribution on $\S^3$ (cf. \cref{lem:unif:S3}).


\subsection{Uniform distribution on $\S^2$ and the use of Hopf fibration}\label{sec:unif:hopf}

Now we use the uniform distribution in $\S^2$, whose points can be written as:
\begin{equation*}
(x,y,z)\in\R^3:
\begin{cases}
x=\sqrt{1-u^2}\cos\phi \\
y=\sqrt{1-u^2}\sin\phi \\
z=u 
\end{cases}
\text{ with }
\begin{cases}
\phi \in [0, 2\pi] \text{ uniformly},\\
u \in [-1, 1] \text{ uniformly}.\\
\end{cases}
\end{equation*}

\begin{proposition}
Let $\nu_M$ be a set of $M$ i.i.d. uniformly distributed points in $\S^2$.
For each point in $\nu_M$ let us consider a set of $k$ roots of unity on the fibre corresponding to the inverse of the point via the Hopf fibration.
Let $\omega_N$ be set of all $N=kM$ points obtained on $\S^3$ by this procedure that we denote by $\HopfUniform{k}{M}$.
Then:
\begin{equation*}
    \Esp_{\omega_N\sim\HopfUniform{k}{M}} \left[\Energy_{0}(\omega_N)\right]
    =
    \frac{-1}{4}N^2+(1-2\log 2)N.
\end{equation*}
\end{proposition}
\begin{proof}
    See \cref{calculo_uniforme_s2}.
\end{proof}

A symmetry argument similar to the one used with the uniform distribution on $\S^3$ can be deduced here. 
Two possibilities arise in this case. 
On one hand, central symmetry in $\S^3$ with points on the Hopf fibres could be used. 
However, since the fibres are great circles, central symmetry would place more points on the same fibre, which would go against the optimal choice of the number of points per fibre. 
Therefore, we chose to use symmetry in $\S^2$ to have better-distributed points on $\S^2$ and calculate the fibres afterwards. 

\begin{proposition}
Consider $M/2 \in \mathbb{N}$ i.i.d.\ uniformly distributed points on $\S^2$, and let $\nu_M$ denote the set of $M$ points consisting of these $M/2$ points together with their antipodal points.
For each point in $\nu_M$ let us consider a set of $k$ roots of unity on the fibre corresponding to the inverse of the point via the Hopf fibration.
Let $\omega_N$ be set of all $N=kM$ points obtained on $\S^3$ by this procedure that we denote by $\HopfUniformAntipodal{k}{M}$.
Then:
\begin{equation*}
    \Esp_{\omega_N \sim \HopfUniformAntipodal{k}{M}} \left[\Energy_{0}(\omega_N)\right]
    =
    \frac{-1}{4}N^2+\left(\frac{7}{2}\left(1-\log2\right)-\log7\right)N.    
\end{equation*}
\end{proposition}

\begin{proof}
The computation of this asymptotic expansion of the energy is shown in \cref{calculo_uniforme_s2}.    
\end{proof}

Comparing linear terms, one sees that the best results among the different uniform distributions are obtained by taking the uniform distribution on $\S^2$, applying central symmetry, and then lifting to $\S^3$ via the Hopf fibration. 
However, none of the point configurations presented here achieves the optimal growth order in the $N\log(N)$ term (the second order term), whereas those presented in the next section do.
\cref{FiguraUniforme} compares the coefficients of the linear term for each uniform distribution.

\section{Determinantal point processes}

In this section, we present the basic theoretical foundations of determinantal point processes in order to enable the computation of the asymptotic expansion of the expected value of their logarithmic energy. 
A comprehensive development is beyond the scope of this work and can be found in \cite{hough_zeros_2009}.

We recall that a Polish space is a separable, completely metrizable topological space.
Let $\Lambda$ be a locally compact, Polish space endowed with a Radon measure $\mu$. 
A simple process of $N$ points in $\Lambda$ is a random variable in $\Lambda^N$ (or equivalently, $N$ random variables simultaneously chosen in $\Lambda$).
Some point processes admit associated intensity functions, defined as follows.

\begin{definition}
Let $\Lambda$ and $\mathcal{X}$ be a space and a simple point process of $N$ points as defined previously. 
Intensity functions are functions, if any exist,
\begin{equation*}
\rho_k : \Lambda^k \rightarrow [0, \infty),
\; 1 \leq k \leq N    
\end{equation*}
such that for any family of pairwise disjoint subsets $D_1, . . . , D_k$ of $\Lambda$, the following holds:
\begin{equation*}
\Esp_{x\sim\mathcal{X}}\left[\prod_{i=1}^k\sharp (x\cap D_i)\right]=\int_{\prod D_i}\rho_k(x_1,...,x_k)d\mu(x_1,...,x_k).    
\end{equation*}
In addition, we shall require that $\rho_{k}(x_1,\ldots,x_k)$ vanish if $x_{i} = x_{j}$ for some $i\neq j$.
\end{definition}

\begin{proposition}\label{prop_intensidad_0}
Let $\Lambda$ and $\mathcal{X}$ be a space and a simple point process of $N$ points with associated intensity functions $\rho_{k}, 1\leq k \leq N$, as in the previous definition. 
For any measurable function $f:\Lambda^k \rightarrow [0, \infty), k \geq 1$, the following holds:
\begin{multline*}
  \Esp_{x\sim\mathcal{X}}\left[\sum_{i_1,...,i_k\text{ distinct}}f(x_{i_1},...,x_{i_k})\right]\\=\int_{y_1,...,y_k\in\Lambda}f(y_1,...,y_k)\rho_k(y_1,...,y_k)d\mu(y_1,...,y_k).
\end{multline*}
\end{proposition}

\begin{proof}
We refer to \cite{etayo_problema_2019}*{Equation 1.13}.
\end{proof}

Some simple point processes with associated intensity functions $\rho_k$, $1 \le k \le N$, have intensity functions of the form
\begin{equation}\label{formula_determinantal}
    \rho_k(x_1,...,x_k)=\det \left(K(x_i,x_j)_{1\leq i, j\leq k}\right)
\end{equation}
for $K: \Lambda\times\Lambda\longrightarrow\C$ a measurable function.
Such processes are called \gls{dpp}. 
The existence of determinantal point processes under suitable assumptions, and with prescribed structural properties, is guaranteed by the Macchi--Soshnikov theorem, which we now recall.

\begin{theorem}[Macchi--Soshnikov]\label{thm-macchi-soshnikov}
Let $\Lambda$ be as described above and let $H \subset L^2(\Lambda)$ be a subspace of dimension $N$.
Then there exists a point process $\mathcal{X}_H$ on $\Lambda$ consisting of $N$ points, whose $k$-point intensity functions are given by
\begin{equation*}
    \rho_k(x_1,...,x_k)
    =
    \det \left(K_{H}(x_i,x_j)_{1\leq i, j\leq k}\right)
\end{equation*}
where $K_H$ denotes the reproducing kernel of $H$ and we say that $\mathcal{X}_H$ is a determinantal point process with kernel $K_H$.
\end{theorem}

\begin{proof}
Proofs of this theorem can be found in \cites{Macchi1975,Soshnikov2000}, and in a formulation closer to the present one in \cite{hough_zeros_2009}*{Theorem 4.5.5}.
\end{proof}

\begin{corollary}\label{cor:constantkernel}
Under the assumptions of \cref{thm-macchi-soshnikov}, the expected number of points satisfies
\begin{equation*}
    N = \mathbb{E}_{X \sim \mathcal{X}_H}[N]
    = \int_{\Lambda} K_H(p,p)\, d\mu(p).
\end{equation*}
In particular, if the diagonal of the kernel $K_H(p,p)$ is constant (independent of $p$), then necessarily
\begin{equation*}
    K_H(p,p)=\frac{N}{{\rm Vol}(\Lambda)},
\end{equation*}
and we say that $K_H$ is homogeneous.
\end{corollary}

\begin{proof}
We take $\varphi \equiv 1$ and substitute it into the previous identity.
\end{proof}

\begin{corollary}\label{prop_intensidad}
Under the assumptions of \cref{thm-macchi-soshnikov}, for any measurable function
$f:\Lambda\times\Lambda\to[0,1)$, one has
\begin{equation*}
    \mathbb{E}_{X\sim \mathcal{X}_H}
    \left[
    \sum_{i\neq j} f(x_i,x_j)
    \right]
    =
    \int_{\Lambda\times\Lambda}
    \bigl(
    K_H(p,p)K_H(q,q)-|K_H(p,q)|^2
    \bigr)
    f(p,q)\,d\mu(p)\,d\mu(q).
\end{equation*}
\end{corollary}

\begin{proof}
    It is a direct consequence of \cref{prop_intensidad_0} and the definition of \gls{dpp}.
\end{proof}

\subsection{Determinantal point processes on $ \mathbb{CP}^1$ and the Hopf fibration}

\begin{proposition}\label{prop:dppHopf}
Let $\nu_r = \{ \mathbf{x}_1,\ldots,\mathbf{x}_r \} \subset \mathbb{CP}^1$ be a set of random points following the distribution of a \gls{dpp} that we denote by $\mathcal{X}^r$.
For each point in $\nu_r$ let us consider a set of $k$ roots of unity on the fibre corresponding to the inverse of the point via the Hopf fibration.
Let $\omega_N$ be set of all $N=kr$ points obtained on $\S^3$ by this procedure that we denote by $\Hopfdpp{k}{r}$.
Then:
\begin{equation*}
\begin{split}
    \Esp_{\omega_N \sim \Hopfdpp{k}{r}} & \left[\Energy_{0}(\omega_N)\right]
    \\
    =   
    -
    rk & \log k
    \\
    -
    \frac{k^2}{2}
    &
    \int_{\mathbb{CP}^1\times\mathbb{CP}^1} 
    \log\left(1+\sqrt{1-<\mathbf{p, q}>^2}\right)\left(K(\mathbf{p,p})^2-|K(\mathbf{p,q})|^2\right)d\mathbf{p}d\mathbf{q},
\end{split}
\end{equation*}
where $K$ is the kernel associated to $\mathcal{X}^r$.
\end{proposition}

\begin{proof}

Let $\mathbf{x}_1, ..., \mathbf{x}_r \in \mathbb{CP}^1$ be obtained through a \gls{dpp} with kernel $K$. 
We choose, for each $\mathbf{x}_i$, an affine representative of unit norm, which we denote by the same symbol.
Let $\underline{i}=\sqrt{-1}$ denote the imaginary unit. Consider 
\begin{equation*}
\mathbf{y}_{ij}=\mathbf{x}_ie^{\underline{i}(\theta_i+\frac{2\pi j}{k})}\in \S^3 \subseteq \mathbb{C}^2,
\; 1\leq i \leq r, \; 1\leq j \leq k,
\end{equation*}
with $\theta_i$ a random phase picked uniformly on $[0, 2\pi]$, as the points in the fibres given by the inverse of the Hopf fibration in its complex form, see \cref{inversaHopfcompleja}.

In the computation of the expected logarithmic energy, two quantities are involved:

\begin{enumerate}
    \item The energy contributed by pairs of points in the same fibre, summed over all the fibres (denote this quantity by $\boldsymbol{J_1}$).
    \item The energy contributed by pairs of points in different fibres, (denote this quantity by $\boldsymbol{J_2}$).
\end{enumerate}


\emph{Computation of $\boldsymbol{J_1}$.}
All the fibres are great circles with $k$ points. Therefore, the quantity $\boldsymbol{J_1}$ equals the energy of one fibre multiplied by the number of fibres.
Furthermore, it is important to note that the distances do not depend on $\theta_i$ since all the points are on the same fibre.
\begin{multline*}
    \boldsymbol{J_1} = \sum^r_{i=1} \frac{1}{2\pi} \int^{2\pi}_0 \Esp_{\theta_i}\left[\sum_{\substack{j_1, j_2=0\\ j_1 \neq j_2}}^{k-1} -\log\left|\left|e^{\underline{i}(\theta_i+\frac{2\pi j_1}{k})}\mathbf{x}_i - e^{\underline{i}(\theta_i+\frac{2\pi j_2}{k})}\mathbf{x}_i\right|\right|  \right] d\theta_i
    \\
    = -r\sum_{\substack{j_1, j_2=0\\ j_1 \neq j_2}}^{k-1}\log \left|e^{\underline{i}\frac{2\pi j_1}{k}} - e^{\underline{i}\frac{2\pi j_2}{k}}\right| = -rk\log k,
\end{multline*}
where we have used the fact that $||\mathbf{x}_i||=1$ and \cref{prop_energia_fibra}.


\emph{Computation of $\boldsymbol{J_2}$.}

\begin{multline*}
    \boldsymbol{J_2} 
    \\
    {\scriptstyle = \frac{1}{4\pi^2}
    \Esp_{\nu_r \sim \mathcal{X}^{r}}
    \left[\sum_{j_1, j_2=0}^{k-1}\sum_{\substack{i_1,i_2=1\\ i_1\neq i_2}}^r\int_0^{2\pi}\int_0^{2\pi} -\log\left|\left|e^{\underline{i}(\theta_{i_1}+\frac{2\pi j_1}{k})}\mathbf{x}_{i_1} - e^{\underline{i}(\theta_{i_2}+\frac{2\pi j_2}{k})}\mathbf{x}_{i_2}\right|\right|d\theta_{i_1}d\theta_{i_2}\right] }
    \\
    =\Esp_{\nu_r \sim \mathcal{X}^{r}}
    \left[\sum_{j_1, j_2=0}^{k-1}\sum_{\substack{i_1,i_2=1\\ i_1\neq i_2}}^r \frac{1}{2\pi}\int_0^{\pi}-\log(2-2<\mathbf{x}_{i_1}, \mathbf{x}_{i_2}>\cos\theta)d\theta\right]
    \\
    =\frac{-k^2}{2}
    \Esp_{\nu_r \sim \mathcal{X}^{r}}
    \left[\sum_{\substack{i_1,i_2=1\\ i_1\neq i_2}}^r \log\left(1+\sqrt{1-<\mathbf{x}_{i_1}, \mathbf{x}_{i_2}>^2}\right)\right],
\end{multline*}
where the last integral was computed using \cref{lema_ruso}.
We conclude with \cref{prop_intensidad}:
\begin{multline*}
    \boldsymbol{J_2} 
    =
    \frac{-k^2}{2}
    \Esp_{\nu_r \sim \mathcal{X}^{r}}
    \left[\sum_{\substack{i_1,i_2=1\\ i_1\neq i_2}}^r \log\left(1+\sqrt{1-<\mathbf{x}_{i_1}, \mathbf{x}_{i_2}>^2}\right)\right]
    \\
    =
    -
    \frac{k^2}{2}
    \int_{\mathbb{CP}^1\times\mathbb{CP}^1} \log\left(1+\sqrt{1-<\mathbf{p, q}>^2}\right)\left(K(\mathbf{p,p})^2-|K(\mathbf{p,q})|^2\right)d\mathbf{p}d\mathbf{q}.
\end{multline*}

\end{proof}

\subsection{The Spherical ensemble}

A natural interpretation of the Spherical ensemble arises from the generalized eigenvalue problem associated with a matrix pencil $(A,B)$, viewed in the complex projective space. More precisely, one considers the points $(\alpha:\beta)\in\CP^1$ such that
\begin{equation*}
    \det(\alpha B-\beta A)=0,
\end{equation*}
and identifies $\CP^1$ with the Riemann sphere.
It was shown by Krishnapur \cite{krishnapur_random_2006} that this construction defines a determinantal point process on $\CP^1$. 
The associated correlation kernel can be described explicitly as follows.
Writing $x=[\mathbf p]$, $y=[\mathbf q]$ with $\mathbf p,\mathbf q\in\C^2\setminus\{0\}$, and with respect to the Fubini--Study volume normalized so that ${\rm Vol}(\CP^1)=\pi$, its correlation kernel is given by
\begin{equation*}
    \bigl|K^{\CP^1}_{r}(x, y)\bigr|
    =
    \frac{r}{\pi}
    \left|
    \left\langle
    \frac{\mathbf p}{\|\mathbf p\|},
    \frac{\mathbf q}{\|\mathbf q\|}
    \right\rangle
    \right|^{\,r-1},
    \qquad
    \mathbf{p},\mathbf{q}\in\C^2\setminus\{0\}.
\end{equation*}
In particular,
\begin{equation*}
    K^{\CP^1}_{r}(x,x)=\frac{r}{\pi}.
\end{equation*}
Via the diffeomorphism mapping the Riemann sphere onto the unit sphere $\S^2$ (see \cref{prop:diffCP1S2}), and using that the surface area measure on $\S^2$ satisfies $\sigma(\S^2)=4\pi$, this construction induces a determinantal point process on $\S^2$ whose correlation kernel is given by
\begin{equation*}
    \bigl|K^{\S^2}_{r}(x,y)\bigr|
    =
    \frac{r}{4\pi}
    \left(\frac{1+\langle x,y\rangle}{2}\right)^{r-1},
    \qquad 
    x,y\in\S^2,
\end{equation*}
where $\langle x,y\rangle$ denotes the Euclidean inner product in $\R^3$. 
In particular,
\begin{equation*}
    K^{\S^2}_{r}(x,x)=\frac{r}{4\pi}.
\end{equation*}
The properties of the Spherical ensemble on $\S^2$ are studied in depth by Alishahi and Zamani \cite{Alishahi2015}, who, in particular compute the expected logarithmic energy of its points.

The kernels $K^{\CP^1}_{r}$ and $K^{\S^2}_{r}$ are related by push-forward under the diffeomorphism identifying $\CP^1$ with $\S^2$, together with the corresponding Jacobian factor between the Fubini--Study volume and the spherical surface measure.


\begin{proposition}\label{prop:01}
Let $\nu_r$ be a set of $r$ random points sampled from the Spherical ensemble $\Sphericalensemble{r}$. 
For each point in $\nu_r$ let us consider a set of $k$ roots of unity on the fibre corresponding to the inverse of the point via the Hopf fibration.
Let $\omega_N$ be the set of all $N=kr$ points obtained on $\S^3$ by this procedure that we denote by $\HopfSphericalensemble{k}{r}$.
Then, the asymptotic expansion of the logarithmic energy of this family of points as $N$ goes to infinity is
\begin{equation*}
    \Esp_{\omega_N \sim \HopfSphericalensemble{k}{r}}
    \left[
    \Energy_{0} (\omega_N)
    \right]
    =
    -\frac{N^2}{4}-N\log k+\frac{\sqrt{\pi}}{4}\sqrt{N}k^{3/2}-\frac{k^2}{4}+k^2O\left(\sqrt{\frac{k}{N}}\right).
\end{equation*}
\end{proposition}

\begin{proof}
See \cref{proof:prop:01}.
\end{proof}

An appropriate choice of $k$ in terms of the number of points $N$ leads to the asymptotic expansion presented in \cref{thm:main}:
\begin{equation*}
-
\frac{N^2}{4}
-
\frac{1}{3}N\log N
+
\frac{N}{3}\left(2+\log\frac{9\pi}{64}\right)
-
\frac{4}{3^{4/3}\pi^{2/3}}N^{2/3}+\mathcal{O}(N^{1/3}). 
\end{equation*}
However, this expansion represents an impossible construction, as it holds for values of $k$ that are not necessarily positive integers. In any case, as explained in \cref{calculodpp}, the constant $\frac{1}{3}\left(2+\log\frac{9\pi}{64}\right)$ can be approximated with arbitrary precision (although this imposes restrictions on the number of points).

\subsection{The Harmonic ensemble}

Let $Y_{\ell m}:\S^2\to\C$ be the standard spherical harmonics of degree $\ell\ge0$ and order $m\in\{-\ell,\dots,\ell\}$, orthonormal with respect to the normalized surface measure on $\S^2$, normalized by $\sigma(\S^2)=4\pi$.
For a given integer $L\ge0$, let $\mathcal H_{\le L}\subset L^2(\S^2)$ denote the $(L+1)^2$-dimensional subspace spanned by all spherical harmonics of degree $\ell\le L$.

The \emph{Harmonic ensemble} on $\S^2$ is defined as the determinantal point process whose correlation kernel $K_L(x,y)$ is the integral kernel of the orthogonal projection onto $\mathcal H_{\le L}$. 
Equivalently, if $\{Y_{\ell m}:0\le \ell\le L,\,-\ell\le m\le \ell\}$ is an orthonormal basis of $\mathcal H_{\le L}$, then
\begin{equation*}
    K^{\S^2}_L(x,y)
    =
    \sum_{\ell=0}^{L}\sum_{m=-\ell}^{\ell}
    Y_{\ell m}(x)\,\overline{Y_{\ell m}(y)},
    \qquad 
    x, y \in\S^2.
\end{equation*}
With respect to the surface area measure $\sigma$ on $\S^2$ one has
\begin{equation*}
    K^{\S^2}_L(x,x)
    =
    \frac{\dim(\mathcal H_{\le L})}{4\pi}
    =
    \frac{(L+1)^2}{4\pi}.
\end{equation*}
By the addition theorem for spherical harmonics, the kernel admits a rotationally invariant expression depending only on the Euclidean inner product $\langle x,y\rangle$:
\begin{equation*}
    K^{\S^2}_L(x,y)
    =
    \frac{1}{4\pi}\sum_{\ell=0}^{L}(2\ell+1)\,
    P_\ell(\langle x,y \rangle)
    =
    \frac{L+1}{4\pi}\,
    P^{(1,0)}_{L}\!\bigl(\langle x,y \rangle\bigr),
\end{equation*}
where $P_\ell$ denotes the Legendre polynomial of degree $\ell$ and
$P^{(1,0)}_{L}$ the Jacobi polynomial with parameters $(1,0)$.

As in the case of the Spherical ensemble, it is natural to consider the Harmonic ensemble intrinsically on the Riemann sphere $\CP^1$.
Let $\Phi:\S^2\to\CP^1$ be the diffeomorphism defined by stereographic projection followed by the affine chart, as in Proposition~\ref{prop:diffCP1S2}.
Pushing forward the process through $\Phi$, and using that the Jacobian satisfies $d\mathrm{vol}_{FS}=\frac14\,\Phi_\#\sigma$, one obtains a determinantal point process on $\CP^1$.

Writing $x=[\mathbf p]$, $y=[\mathbf q]$ with
$\mathbf p,\mathbf q\in\C^2\setminus\{0\}$, the correlation kernel of the Harmonic ensemble on $\CP^1$, with respect to the Fubini--Study volume normalized by ${\rm Vol}(\CP^1)=\pi$, is given explicitly by
\begin{equation*}
    K^{\CP^1}_L(x,y)
    =
    \frac{L+1}{\pi}\,
    P^{(1,0)}_{L}\!\left(
    2\left|\left\langle
        \frac{\mathbf p}{\|\mathbf p\|},
        \frac{\mathbf q}{\|\mathbf q\|}
        \right\rangle
        \right|^2
    -1
    \right).
\end{equation*}
In particular,
\begin{equation*}
    K^{\CP^1}_L(x,x)=\frac{(L+1)^2}{\pi}.
\end{equation*}
This expression shows that, in contrast with the Spherical ensemble, the Harmonic ensemble on $\CP^1$ is a purely radial determinantal process: the kernel depends only on the invariant quantity $\left|\left\langle
        \frac{\mathbf p}{\|\mathbf p\|},
        \frac{\mathbf q}{\|\mathbf q\|}
        \right\rangle
        \right|^2$.

The Harmonic ensemble on projective spaces, and in particular on $\CP^1$, is studied by Anderson et al.\ \cite{Anderson_Dostert_Grabner_Matzke_Stepaniuk_2023}, who obtain the same correlation kernel up to a different normalization.

\begin{proposition}\label{prop:02}
Let $\nu_r$ be a set of $r$ random points sampled from the Harmonic ensemble, $\Harmonicensemble{r}$. 
For each point in $\nu_r$ let us consider a set of $k$ roots of unity on the fibre corresponding to the inverse of the point via the Hopf fibration.
Let $\omega_N$ be set of all $N=kr$ points obtained on $\S^3$ by this procedure that we denote by $\HopfHarmonicensemble{k}{r}$.
Then, the asymptotic expansion of the logarithmic energy of this family of points as $N$ goes to infinity is
\begin{equation*}
    \Esp_{\omega_N \sim \HopfHarmonicensemble{k}{r}}
    \left[
    \Energy_{0}(\omega_N)
    \right]
    =
    -
    \frac{N^2}{4}
    -
    N\log (k)
    +
    \frac{k^2 \sqrt{r} \log (r)}{4\pi }
    +
     k^2 
     \mathcal{O}\left( \sqrt{r} \right)
     .
\end{equation*}
\end{proposition}

\begin{proof}
See \cref{proof:prop:02}.
\end{proof}

\begin{corollary}\label{cor:Harm:asymp}
Under the assumptions of \cref{prop:02}, if we take $k=\left\lfloor\frac{\sqrt{r}}{\log(r)}\right\rfloor$ points, the asymptotic expansion of the logarithmic energy of this family of points as $N$ goes to infinity is
\begin{equation*}
    \Esp_{\omega_N \sim \HopfHarmonicensemble{k}{r}}
    \left[
    \Energy_{0}(\omega_N)
    \right]
    =
    -
    \frac{N^2}{4}
    - 
    \frac{1}{3}N\log(N) + \frac{2}{3}N\log\log(N)
    +
    \mathcal{O}(N).
\end{equation*}
\end{corollary}

\begin{proof}
    See \cref{proof:prop:02}.
\end{proof}

\subsection{Spherical vs Harmonic ensemble}

In order to facilitate a direct comparison between the Spherical and the Harmonic ensembles, we adopt a unified notation for the number of points. 
Let
\begin{equation*}
    r = (L+1)^2.
\end{equation*}
Then, the Harmonic ensemble on $\S^2$ consists of $r$ points, where $r$ is necessarily a perfect square. 
In contrast, the Spherical ensemble can be defined for an arbitrary positive integer $r\in\mathbb N$ of points, with no arithmetic restriction. 
For the reader’s convenience, we summarize the relevant kernel formulas in
\cref{tab:kernels-comparison}.

\begin{table}[ht]
\centering
\caption{Correlation kernels of the Spherical and Harmonic ensembles on $\S^2$ and $\CP^1$, where, for the points in $\CP^1$, $x=[\mathbf p]$, $y=[\mathbf q]$ with $\mathbf p,\mathbf q\in\C^2\setminus\{0\}$.}
\label{tab:kernels-comparison}
\begin{tabular}{c c c}
\hline
Ensemble & Ambient space & Correlation kernel \\
\hline
Spherical & $\CP^1$ &
$\displaystyle
\bigl|K^{\CP^1}_{r}(x,y)\bigr|
=\frac{r}{\pi}
\left|
\left\langle
\frac{\mathbf p}{\|\mathbf p\|},
\frac{\mathbf q}{\|\mathbf q\|}
\right\rangle
\right|^{r-1}
$ \\[2mm]

Spherical & $\S^2$ &
$\displaystyle
\bigl|K^{\S^2}_{r}(x,y)\bigr|
=\frac{r}{4\pi}
\left(\frac{1+\langle x,y\rangle}{2}\right)^{r-1}
$ \\[2mm]

Harmonic & $\CP^1$ &
$\displaystyle
K^{\CP^1}_L(x,y)
=
\frac{\sqrt{r}}{\pi}\,
P^{(1,0)}_{L}\!\left(
2\left|
\left\langle
\frac{\mathbf p}{\|\mathbf p\|},
\frac{\mathbf q}{\|\mathbf q\|}
\right\rangle
\right|^2
-1
\right)
$ \\[2mm]

Harmonic & $\S^2$ &
$\displaystyle
K^{\S^2}_L(x,y)
=
\frac{\sqrt{r}}{4\pi}\,
P^{(1,0)}_{L}\!\bigl(\langle x,y\rangle\bigr)
$ \\
\hline
\end{tabular}
\end{table}

This table highlights the structural differences between both ensembles. While
the Spherical ensemble leads to power-law kernels, the Harmonic ensemble exhibits
polynomial correlations governed by Jacobi polynomials. In both cases, the
kernels are rotationally invariant and admit natural formulations both on the
sphere and on the complex projective line.



\section{Proof of the main results}

\subsection{Auxiliary results}


\begin{proposition}\label{prop_energia_fibra}
The logarithmic energy of the $k$ roots of unity on $\S^1$ is $-k\log k.$
\end{proposition}

\begin{proof}
It is an easy computation, the interested reader can check it in \cite{brauchart_riesz_2009}.
\end{proof}


\begin{lemma}\label{lema_ruso}
Let $a\geq|b|>0$, then
$$\int_0^\pi\log(a+b\cos x)dx=\pi\log\left(\frac{a+\sqrt{a^2-b^2}}{2}\right).$$
\end{lemma}

\begin{proof}
Formula obtained from \cite{gradshtein_table_2007}*{Formula 4.224}.
\end{proof}

\subsection{Points coming from the uniform distribution}\label{calculo_uniforme_s2}

Points uniformly distributed on $\S^2$ take the form:
\begin{equation}\label{eq:02}
    (x,y,z)\subset\R^3 :
    \begin{cases}
    x=\sqrt{1-u^2}\cos\phi \\
    y=\sqrt{1-u^2}\sin\phi \\
    z=u 
    \end{cases}
    \text{ with }
    \begin{cases}
    \phi \in [0, 2\pi] \text{ uniform},\\
    u \in [-1, 1] \text{ uniform}.\\
    \end{cases}
\end{equation}
The image of such points by the inverse of the Hopf fibration given in \cref{inversa_hopf} is
\begin{equation*}
    \mathbf{x}
    =
    \frac{1}{\sqrt{2}}\left(\sqrt{1+u}\cos t, \sqrt{1+u}\sin t, \sqrt{1-u}\sin (t-\phi),\sqrt{1-u}\cos (t-\phi)\right),
\end{equation*}
with $t\in[0,2\pi)$.
Hence, the distance between two points, $\mathbf{x}_i$ and $\mathbf{x}_j$, is given by
\begin{multline*}
||\mathbf{x}_i-\mathbf{x}_j||
=
\left[2-\sqrt{1+u_i}\sqrt{1+u_j}\cos(t_i-t_j)
\right.\\\left.
-\sqrt{1-u_i}\sqrt{1-u_j}\cos((t_i-t_j)-(\phi_i-\phi_j))\right]^{1/2},
\end{multline*}
with $t_i,t_j\in[0,2\pi)$.
Since $\S^3$ is a 2-point homogeneous space, we can take $\mathbf{x}_i=(1, 0, 0, 0)$ to simplify the calculations. Thus, the distance is
\begin{equation*}
||\mathbf{x}_i-\mathbf{x}_j||=\sqrt{2-\sqrt{2(1+u_j)}\cos t_j}.    
\end{equation*}
Then, the expectation of the logarithm of the inverse of the distance of a pair of points belonging to different fibres is given by
\begin{multline*}
    \Esp_{\phi_j, u_j, t_j}
    [-\log||(1,0,0,0)-\mathbf{x}_j||]
    \\
    =\frac{-1}{4\pi}\int_{-1}^1\int_0^{2\pi}\log\left(\sqrt{2-\sqrt{2(1+u_j)}\cos t_j}\right)dt_jdu_j
    \\
    =\frac{-1}{4}\int_{-1}^1\log\left(1+\frac{\sqrt{2-2u}}{2}\right)du=\frac{-1}{4},
\end{multline*}
where \cref{lema_ruso} has been used to pass from the second line to the third.

For the computation of the expected logarithmic energy, we have to consider all the pairs of points belonging to the same fibre.
In this case, the energy will correspond to the roots of unity, so, according to \cref{prop_energia_fibra}, we have
\begin{equation*}
    \Energy_{=fibre}=-\log k
\end{equation*}
for each point in the fibre. 
Thus the energy will be
\begin{multline*}
    \Esp_{\substack{\phi_j, u_j, v_j, \\  1\leq j \leq M}}
    [\Energy_{0}(\omega_N)]
    =
    N[E_{\neq fibre}+ E_{=fibre}]
    \\
    =
    N\left[\frac{-1}{4}(N-k)-\log k\right]=\frac{-1}{4}N^2+\frac{1}{4}Nk-N\log k,
\end{multline*}
where $\omega_N$ is the set of all $N=kM$ points obtained on $\S^3$ by the procedure $\HopfUniform{k}{M}$.
Let us consider whether there exists an optimal value of $k$ that minimizes it. 
$\Esp_{\substack{\phi_j, u_j, v_j, \\  1\leq j \leq M}} [\Energy_{0}(\omega_N)]$ is a continuous and differentiable function of $k$ on $[1,\infty)$.
To find the optimal $k$, we set the derivative of the expected energy (with respect to $k$) equal to zero:
\begin{equation*}
  \frac{\partial}{\partial k}\Esp_{\substack{\phi_j, u_j, v_j, \\  1\leq j \leq M}} [\Energy_{0}(\omega_N)]
  =
  \frac{1}{4}N-\frac{N}{k}.  
\end{equation*}
The derivative vanishes in $[1,\infty)$ if and only if $k=4$, and the second derivative with respect to $k$ is positive in this interval, so one can conclude that the function attains its minimum when $k=4$.
Therefore, the optimal number of points in each fibre is $4$. 
This result may seem counter-intuitive, because regardless of the number of points on $\S^2$, the optimal number of points per fibre is always 4. The asymptotic expansion of the logarithmic energy for this choice is then:
\begin{equation*}
    \Esp_{\substack{\phi_j, u_j, v_j, \\  1\leq j \leq M}} [\Energy_{0}(\omega_N)]
    =
    \frac{-1}{4}N^2+(1-2\log 2)N.
\end{equation*}
%


\subsubsection*{Calculation with antipodal points}

As explained in \cref{sec:unif:hopf}, we impose symmetry on $\S^2$ to obtain a better distribution there, and then compute the points along the fibres.

Following the notation from \cref{eq:02}, points that are antipodal to those given in \eqref{eq:02} have the following form:
\begin{equation*}
(x,y,z)\in\R^3 :
\begin{cases}
x=\sqrt{1-u^2}\cos(\phi+\pi) \\
y=\sqrt{1-u^2}\sin(\phi+\pi) \\
z=-u 
\end{cases}
\text{ with }
\begin{cases}
\tilde{\phi}=\phi+\pi,\\
\tilde{u}=-u.\\
\end{cases}  
\end{equation*}
where $\phi$ and $u$ are distributed as in \eqref{eq:02}.
From these points and by means of \cref{inversa_hopf} we obtain the fibres
\begin{equation*}
\mathbf{\tilde{x}}=\frac{1}{\sqrt{2}}\left(
\sqrt{1-u}\cos \tilde{t}, \sqrt{1-u}\sin \tilde{t}, \sqrt{1+u}\sin (\tilde{t}-\phi-\pi),\sqrt{1+u}\cos (\tilde{t}-\phi-\pi)\right),
\end{equation*}
where $\tilde{t}=\frac{2\pi l}{k}+\theta$, with $\theta\in[0,2\pi)$ distributed uniformly randomly, represents the angle within the fibre.
Therefore, the expected sum of the logarithm of the inverse of the distances between a point and all the points on the \emph{antipodal fibre} is
\begin{multline*}
\Energy_{\substack{antipodal \\ \textit{fibre}}}
=
-\frac{1}{2}
\Esp_{\theta}
\left[
\sum_{i=1}^{k}\log||\mathbf{x}-\mathbf{\tilde{x}}||^2
\right]
\\
=
-
\frac{1}{2}
\Esp_{\theta}
\left[
\sum_{i=1}^{k}\log
\left(
2-\sqrt{1-u^2}\left[\cos(t-\tilde{t})+\cos(t-\tilde{t}+\pi)\right]
\right)
\right]
\\
=
-
\frac{1}{2}
\Esp_{\theta}
\left[
\sum_{i=1}^{k}\log
(2)
\right]
=
\frac{-\log(2)}{2}k.
\end{multline*}
The asymptotic expansion of the expected logarithmic energy will then be
\begin{multline*}    
\Esp_{\substack{\phi_j, u_j, v_j, \\  1\leq j \leq M/2}}
[\Energy_{0}(\omega_N)]
\\
=
N\left(
\Energy_{=fibre}
+ 
\Energy_{\substack{antipodal \\ \textit{fibre}}}
+
\Esp_{\substack{\phi_j, u_j, v_j, \\  1\leq j \leq N}}\left[\Energy_{\substack{\neq fibre \\ \textit{no antipodal}}}\right]
\right)
\\
=N\left(-\log k+\frac{-\log 2}{2}k+\frac{-1}{4}(N-2k)\right)
\\=
\frac{-1}{4}N^2+\frac{Nk}{2}(1-\log2)-N\log k.
\end{multline*}
Similarly to the previous case, it is now relevant to consider the optimal choice of the number of points per fibre, denoted as $k$.
Taking the derivative with respect to $k$ we have
$$
  \frac{\partial}{\partial k}\Esp_{\substack{\phi_j, u_j, v_j, \\  1\leq j \leq M/2}} [\Energy_{0}(\omega_N)]
=
\frac{N}{2}(1-\log2)-\frac{N}{k}.$$
Setting it equal to zero:
$$
  \frac{\partial}{\partial k}\Esp_{\substack{\phi_j, u_j, v_j, \\  1\leq j \leq M/2}} [\Energy_{0}(\omega_N)]
=
0 \Rightarrow k=\frac{2}{1-\log2}\approx 6.51...$$
Furthermore, the second derivative is positive.
Therefore, the minimum occurs at either $k=6$ or $k=7$. Substituting and considering the linear term:
\begin{itemize}
    \item $k=6 \longrightarrow 3(1-\log2)-\log6\approx-0.8712...$
    \item $k=7 \longrightarrow \frac{7}{2}(1-\log2)-\log7\approx-0.8719...$
\end{itemize}
Thus, the minimum is achieved at $k=7$, and the asymptotic expansion is
\begin{equation*}
\Esp_{\substack{\phi_j, u_j, v_j, \\  1\leq j \leq M/2}}
[\Energy_{0}(\omega_N)]
 =
 \frac{-1}{4}N^2+\left(\frac{7}{2}(1-\log2)-\log 7\right)N.
\end{equation*}
Again, the optimal $k$ is a fixed number ($7$ in this case), independent of the total number of points.


\subsection{Points coming from determinantal point processes}
Both of the \gls{dpp} that we are working with are homogeneous and  satisfy
\begin{equation*}
    |K(\mathbf{p, q})|^2 = K(\mathbf{p, p})^2 f(|<\mathbf{p, q}>|).
\end{equation*}

\begin{proposition}\label{prop:dpp:iso}
Let $\nu_r = \{ \mathbf{x}_1,\ldots,\mathbf{x}_r \} \subset\mathbb{CP}^1$ be a set of random points following the distribution of a \gls{dpp} that we denote by $\mathcal{X}^r$, with associated reproducing kernel $K$ satisfying
\begin{equation*}
 |K(\mathbf{p, q})|^2 = K(\mathbf{p, p})^2 f(|<\mathbf{p, q}>|).
\end{equation*}
For each point in $\nu_r$ let us consider a set of $k$ roots of unity on the fibre corresponding to the inverse of the point via the Hopf fibration.
Let $\omega_N$ be set of all $N=kr$ points obtained on $\S^3$ by this procedure that we denote by $\Hopfdpp{k}{r}$.
Then:
\begin{multline*}
    \Esp_{\omega_N \sim \Hopfdpp{k}{r}} \left[\Energy_{0}(\omega_N)\right]
    \\
    =   
    -
    rk\log k
    -
    \frac{k^2r^2}{4}
    +
    k^2r^2
    \int_{0}^\infty \log\left(1+\sqrt{1-\frac{1}{1+t^2}}\right)
    f\left(\frac{1}{\sqrt{1+t^2}}\right)
    \frac{t dt}{(1+t^2)^2}
    .
\end{multline*}
\end{proposition}

\begin{proof}
We use \cref{prop:dppHopf} and follow the notation of its proof, together with \cref{cor:constantkernel}, to state that
\begin{multline*}
    \boldsymbol{J_2}=\frac{-k^2}{2}\int_{\mathbb{CP}^1\times\mathbb{CP}^1} \log\left(1+\sqrt{1-<\mathbf{p, q}>^2}\right)\left(K(\mathbf{p,p})^2-|K(\mathbf{p,q})|^2\right)d\mathbf{p}d\mathbf{q}
    \\
    =\frac{-k^2r^2}{2\pi^2}\int_{\mathbb{CP}^1\times\mathbb{CP}^1} 
    \log\left(1+\sqrt{1-<\mathbf{p, q}>^2}\right)
    \left(1-f(|<\mathbf{p,q}>|)\right)d\mathbf{p}d\mathbf{q}.
\end{multline*}
Since $\mathbb{CP}^1$ is 2-point homogeneous and its volume is $\pi$, in order to compute this quantity, we can fix $\mathbf{p}$ to be $\mathbf{e_1}=(1,0,0)$:
$$\boldsymbol{J_2}=\frac{-k^2r^2}{2\pi}\int_{\mathbb{CP}^1} \log\left(1+\sqrt{1-<\mathbf{e_1, q}>^2}\right)\left(1-f(|<\mathbf{e_1,q}>|)\right)d\mathbf{q}.$$
Let $\psi$ be the map:
\begin{align*}
    \psi: \mathbb{C} & \longrightarrow \mathbb{CP}^1\\
    \mathbf{z} & \longmapsto (\mathbf{z},1),
\end{align*}
then its normal Jacobian is
\begin{equation*}
NJac(\psi)(\mathbf{z})
=
\left(\frac{1}{1+||\mathbf{z}||^2}\right)^{2},    
\end{equation*}
and so:
\begin{multline*}
    \boldsymbol{J_2}=\frac{-k^2r^2}{2\pi}\int_{\mathbb{C}}
    \left[
    \log\left(1+\sqrt{1-\left<\mathbf{e_1}, \frac{(1,\mathbf{z})}{\sqrt{1+||\mathbf{z}||^2}}\right>^2}\right)
    \right.
    \\
    \left.
    \times
    \left(1-f\left(\left|\left<\mathbf{e_1},\frac{(1,\mathbf{z})}{\sqrt{1+||\mathbf{z}||^2}}\right>\right|\right)\right)\frac{1}{(1+||\mathbf{z}||^2)^2}d\mathbf{z}
    \right]
     \\
     =\frac{-k^2r^2}{2\pi}\int_{\mathbb{C}} \log\left(1+\sqrt{1-\frac{1}{1+||\mathbf{z}||^2}}\right)
     \left(1-f\left(\frac{1}{\sqrt{1+||\mathbf{z}||^2}}\right)\right)\frac{1}{(1+||\mathbf{z}||^2)^2}d\mathbf{z}.
\end{multline*}
Changing to polar coordinates:
\begin{multline*}
    \boldsymbol{J_2}
    =
    -k^2r^2 \left[
    \int_{0}^\infty \log\left(1+\sqrt{1-\frac{1}{1+t^2}}\right)
    \frac{tdt}{(1+t^2)^2}
    \right.\\\left.
    -
    \int_{0}^\infty \log\left(1+\sqrt{1-\frac{1}{1+t^2}}\right)
    f\left(\frac{1}{\sqrt{1+t^2}}\right)
    \frac{t}{(1+t^2)^2}
    dt
    \right].
\end{multline*}
Where the first integral is
\begin{multline}\label{eq:05}
    \int_{0}^\infty \log\left(1+\sqrt{1-\frac{1}{1+t^2}}\right) \frac{tdt}{(1+t^2)^2} 
    \\
    = 
    \int_{0}^\infty \log\left(\sqrt{1+t^2} + t \right) \frac{tdt}{(1+t^2)^2} 
    -\frac{1}{2}
    \int_{0}^\infty \log\left(1+t^2 \right) \frac{tdt}{(1+t^2)^2}.
\end{multline}
For the first integral of \eqref{eq:05}, the hyperbolic change of variables $t = \sinh u$, and the posterior integration by parts, give us
\begin{equation*}
    \int_{0}^\infty \log\left(\sqrt{1+t^2} + t \right) \frac{tdt}{(1+t^2)^2} 
    =
    \int_{0}^\infty \frac{\sinh u}{\cosh^3u} u du
    =
    \frac{1}{2}.
 \end{equation*}
The second integral of \eqref{eq:05} can be computed using the change of variables $u=1+t^2$ and integrating by parts:
\begin{equation*}
    -\frac{1}{2}
    \int_{0}^\infty \log\left(1+t^2 \right) \frac{tdt}{(1+t^2)^2}
    =
    -\frac{1}{4}
    \int_{1}^\infty \log\left( u \right) \frac{du}{u^2}
    =
    -\frac{1}{4}
    \int_{1}^\infty 
    \frac{du}{u^2}
    =
    -\frac{1}{4}.
\end{equation*}
Hence we conclude
\begin{equation*}
    \int_{0}^\infty \log\left(1+\sqrt{1-\frac{1}{1+t^2}}\right) \frac{tdt}{(1+t^2)^2}
    =
    \frac{1}{4}.
\end{equation*}
\end{proof}

\subsection{Points coming from the Spherical ensemble}

To compute the expected logarithmic energy, we use \cref{prop:dpp:iso} with
\begin{equation*}
    f(|<\mathbf{p, q}>|)
    =
    |<\mathbf{p, q}>|^{2(r-1)},
\end{equation*}
where $\mathbf{p, q}$ are unit norm representatives and therefore,
\begin{equation*}
    f\left( \frac{1}{\sqrt{1+t^2}} \right)
    =
    \frac{1}{(1+t^2)^{(r-1)}}
    .
\end{equation*}
Hence, integrating by parts, we obtain
\begin{multline}\label{eq:06}
    I_{\text{\eqref{eq:06}}}
    =
    \int_{0}^\infty \log\left(1+\sqrt{1-\frac{1}{1+t^2}}\right)
    \frac{tdt}{(1+t^2)^{r+1}} 
    \\
    =
    \frac{1}{2r}
    \left[
    \int_0^\infty \frac{\sqrt{1+t^2}}{(t^2+1)^{r+1}}dt
    -
    \int_0^\infty \frac{t}{(t^2+1)^{r+1}}dt    
    \right],
\end{multline}
where the first integral is a Beta integral
\begin{equation*}
    \int_0^\infty \frac{\sqrt{1+t^2}}{(t^2+1)^{r+1}}dt
    =
    \frac{\sqrt{\pi}\Gamma(r)}{2\Gamma\left( r+\frac{1}{2}\right)},
\end{equation*}
and for the second integral, we use the change of variables $t=\tan\theta$ and then the change of variables $u=\cos\theta$ 
\begin{equation*}
    -
    \int_0^\infty \frac{t}{(t^2+1)^{r+1}}dt
    =
    -
    \int_{0}^{\pi/2} 
    \sin\theta
    \left(\cos\theta\right)^{2r-1}
    d\theta
    =
    -
    \int_{0}^{1} 
    u^{2r-1}
    du
    =
    -
    \frac{1}{2r}.
\end{equation*}
Summarizing, we have that 
\begin{equation*}
    I_{\text{\eqref{eq:06}}}
    =
    \frac{1}{2r}
    \frac{\sqrt{\pi}\Gamma(r)}{2\Gamma\left( r+\frac{1}{2}\right)}
    -
    \frac{1}{2r}
    \frac{1}{2r}
    =
    \frac{\sqrt{\pi}\Gamma(r+1)-\Gamma\left( r+\frac{1}{2}\right)}{4r^2\Gamma\left( r+\frac{1}{2}\right)}.
\end{equation*}
Thus, we have
$$
\boldsymbol{J_2}
=
-k^2r^2\left[
\frac{1}{4}
-
\frac{\sqrt{\pi}\Gamma(r+1)-\Gamma\left( r+\frac{1}{2}\right)}{4r^2\Gamma\left( r+\frac{1}{2}\right)}
\right]
$$
and then
\begin{equation}\label{E_dpp_sin_simplificar}
    \Esp_{\omega_N \sim \HopfSphericalensemble{k}{r}}
    \left[
    \Energy_{0}(\omega_N)
    \right]
    =
    \boldsymbol{J_1}+\boldsymbol{J_2}
    =
    -
    \frac{k^2r^2}{4}
    -
    rk\log k
    +
    k^2\frac{\sqrt{\pi}\Gamma(r+1)}{4\Gamma\left( r+\frac{1}{2}\right)}
    -
    \frac{k^2}{4}    
    . 
\end{equation}
%
\subsubsection{Proof of \cref{prop:01}}\label{proof:prop:01}

Stirling formula states that
$$
n!
=
\sqrt{2\pi n}\left(\frac{n}{e}\right)^n+O\left(\frac{n^{n-1/2}}{e^n}\right)
$$
for $n$ large enough. 
If we use the formula, we have, after carefully manipulation,
\begin{equation*}
    \frac{\sqrt{\pi}\Gamma(r+1)}{\Gamma\left( r+\frac{1}{2}\right)}
    =
    \frac{2^{2r}(r!)^2}{(2r)!}
=\sqrt{\pi r}
+O\left(\frac{1}{\sqrt{r}}\right).
\end{equation*}
Substituting in the asymptotic expansion of \cref{E_dpp_sin_simplificar} we have
$$
    \Esp_{\omega_N \sim \HopfSphericalensemble{k}{r}}
    \left[
    \Energy_{0}(\omega_N)
    \right]
=
-\frac{k^2r^2}{4}-rk\log k-\frac{k^2}{4}+\frac{k^2\sqrt{\pi}}{4}\sqrt{r}+k^2O\left(\frac{1}{\sqrt{r}}\right)
$$
and taking into account that $N=rk$:
\begin{equation}\label{E_dpp_simplificada}
    \Esp_{\omega_N \sim \HopfSphericalensemble{k}{r}}
    \left[
    \Energy_{0}(\omega_N)
    \right]
=
    -\frac{N^2}{4}-N\log k+\frac{\sqrt{\pi}}{4}\sqrt{N}k^{3/2}-\frac{k^2}{4}+k^2O\left(\sqrt{\frac{k}{N}}\right),
\end{equation}
as desired.

\subsubsection{Optimal choice of \textit{k}}\label{calculodpp}

A first approximation could be to take the derivative of \cref{E_dpp_simplificada} with respect to the parameter $k$, yielding:
\begin{equation*}
\frac{\partial}{\partial k}
    \Esp_{\omega_N \sim \HopfSphericalensemble{k}{r}}
    \left[
    \Energy_{0}(\omega_N)
    \right]
=
-\frac{N}{k}+\frac{3\sqrt{\pi}}{8}\sqrt{N}\sqrt{k}-\frac{k}{2}.    
\end{equation*}
However, setting this expression equal to zero provides solutions that are quite complex and do not lead to a clear asymptotic expansion of the energy. Therefore, after exploring several ways to select $k$, we choose  $k=CN^\alpha$ as a constant times a power of $N$ ($\alpha\in[0,1]$).

If we take $k=CN^\alpha$, then
\begin{multline*}
    \Esp_{\omega_N \sim \HopfSphericalensemble{k}{r}}
    \left[
    \Energy_{0}(\omega_N)
    \right]
     \\
    =-\frac{N^2}{4}+N\left[-\log (CN^\alpha)+\frac{\sqrt{\pi}}{4}C^{3/2}N^\frac{3\alpha -1}{2}-\frac{C^2N^{2\alpha-1}}{4}
\right]
+\mathcal{O}(N^{\frac{1+3\alpha}{2}})
.
\end{multline*}
We want the term in square brackets to be as small as possible. It is important to note several facts for this purpose. Firstly, we have $\frac{3\alpha -1}{2}\geq 2\alpha-1\quad\forall \alpha\in[0,1]$. 
Therefore, the second term is bigger than the third term if $\frac{3\alpha -1}{2} > 0$. Since this term is positive, we aim to have $\frac{3\alpha -1}{2} \leq 0$, which occurs when $\alpha \in [0, 1/3]$. Consequently, the dominant term will be the first one. Expanding this term we obtain
$$-\log(CN^\alpha)=-\log C-\alpha \log N,$$
which will clearly be smaller the larger $\alpha$ is. Therefore, the optimal value will be $\alpha=1/3$. 

Once the value of $\alpha$ is fixed, minimizing the linear term again provides a value for the constant:
$$C=\frac{4}{3^{2/3}\pi^{1/3}}.$$
Thus, we obtain an expansion for the logarithmic energy:
\begin{multline*}
    \Esp_{\omega_N \sim \HopfSphericalensemble{k}{r}}
    \left[
    \Energy_{0}(\omega_N)
    \right]
\\
=
-\frac{N^2}{4}-\frac{1}{3}N\log N+\frac{N}{3}\left(2+\log\frac{9\pi}{64}\right)-\frac{4}{3^{4/3}\pi^{2/3}}N^{2/3}+\mathcal{O}(N^{1/3}).
\end{multline*}
The coefficient of the linear term in this expansion is $\frac{1}{3}\left(2+\log\frac{9\pi}{64}\right)=0.394...$.
However, this value of $C$ would result in an impossible construction because we would have:
\begin{equation}
    k=CN^\alpha=\frac{4}{3^{2/3}\pi^{1/3}}N^{1/3}\notin \mathbb{N}.
    \label{cte}
\end{equation}
Therefore, we seek to approximate the value of the constant $C$ so that both $k$ and $N$ are positive integers.

From \cref{cte}, we have:
$$N=\frac{3^2\pi}{4^3}k^3$$
and we have $\frac{3^2\pi}{4^3}\approx 0.44178\ldots$. 
The approximation can be made arbitrarily close (at the cost of restricting possible $k$  values). 
For example, if we approximate $\frac{3^2\pi}{4^3}$ by $\frac{1}{2}$, then we have:
$$N=\frac{1}{2}k^3 \text{ and taking }k=2k'\text{ and } N=4{k'}^3\in\mathbb{N},$$
the linear-term coefficient becomes
\begin{equation*}
    \frac{\sqrt{2\pi}}{4}-\log\sqrt[3]{2}=0.396\ldots
    .
\end{equation*}
However, as mentioned earlier, it can be arbitrarily accurate. 
For example, approximating $\frac{3^2\pi}{4^3}$ by $\frac{2}{5}$ results in a linear term of $\frac{\sqrt{\pi}}{4}\sqrt[2]{\frac{5}{2}}-\log\sqrt[3]{\frac{5}{2}}=0.395...$.

\subsubsection{Proof of \cref{main:cor}}

Let $K_S = \frac{3^2\pi}{4^3}$.
For each $i\in\mathbb N$, define
\begin{equation*}
    B_i = 10^i,
\qquad
A_i = \lfloor K_S 10^i \rfloor.
\end{equation*}
Then
\begin{equation*}
    \frac{A_i}{B_i} \le K_S
    \quad\text{and}\quad
    0 \le K_S - \frac{A_i}{B_i} < \frac{1}{10^i},
\end{equation*}
so that $\frac{A_i}{B_i} \longrightarrow K_S$.
Set $k_i = 10^i$, $r_i=\frac{A_i}{B_i}k_i^2$ and define
\begin{equation*}
    n_i =k_ir_i =  A_i k_i^2.
\end{equation*}
Then $n_i\to\infty$, and since $A_i \sim K_S k_i$, we have
\begin{equation*}
    n_i \sim K_S k_i^3,
\qquad
k_i \sim \left(\frac{n_i}{K_S}\right)^{1/3}.
\end{equation*}
By \cref{thm:main}, the coefficient of the linear term is
\begin{equation*}
    f \left(\frac{A_i}{B_i}\right),
    \qquad
    f(x)=\frac{\sqrt{\pi}}{4}x^{3/2}-\log x.
\end{equation*}
Since $f$ is smooth in a neighbourhood of $K_S>0$ and 
$\frac{A_i}{B_i}\to K_S$, we obtain
\begin{equation*}
    f\left(\frac{A_i}{B_i}\right)
    =
    f(K_S)+o(1)
    =
    C_S+o(1).
\end{equation*}
Substituting into the expansion provided by \cref{thm:main} yields
\begin{equation*}
    \min_{\omega_{n_i}\subset \mathbb S^3}
    \Energy_0(\omega_{n_i})
    \le
    -\frac{n_i^2}{4}
    -\frac{1}{3}n_i\log n_i
    +
    C_S n_i
    +
    o(n_i),
\end{equation*}
which concludes the proof.

The above construction is not unique: many other infinite subsequences can be obtained by choosing any sequence of rational approximations converging to $K_S$. The use of powers of $10$ is merely for convenience.

\subsection{Points coming from the Harmonic ensemble}\label{proof:prop:02}

To compute the expected logarithmic energy, we use \cref{prop:dpp:iso} with
\begin{equation*}
    f(|<\mathbf{p, q}>|)
    =
    \left(
    \frac{
    P^{(1,0)}_{L}\!\left(
    2\left|\left\langle
        \mathbf{p, q}
        \right\rangle
        \right|^2
    -1
    \right)
    }{L+1}
    \right)^{2}
    ,
\end{equation*}
where $\mathbf{p, q}$ are unit norm representatives. 
Therefore,
\begin{equation*}
    f\left( \frac{1}{\sqrt{1+t^2}} \right)
    =
    \left(
    \frac{
    P^{(1,0)}_{L}\!\left(
    \frac{1-t^2}{1+t^2} 
    \right)
    }{L+1}
    \right)^{2}
    .
\end{equation*}
Hence, we have to compute
\begin{multline*}
    I_{L}
    =
    \frac{1}{(L+1)^2}
    \int_{0}^\infty \log\left(1+\sqrt{1-\frac{1}{1+t^2}}\right)
    P_{L}^{(1,0)} \left(  \frac{1-t^2}{1 + t^2}\right)^2
    \frac{1}{(1+t^2)^2}tdt
    \\
    =
    \frac{1}{4(L+1)^2}
    \int_{0}^{\pi}
    \log\bigl(1+\sin(\theta/2)\bigr)\,
    \bigl(P_L^{(1,0)}(\cos\theta)\bigr)^2
    \sin\theta\,d\theta .
\end{multline*}
where we have used the change of variables $t=\tan(\theta/2)$.
Define
\begin{equation*}
    \ell(\theta):=\log\bigl(1+\sin(\theta/2)\bigr).
\end{equation*}
For Jacobi polynomials $P_n^{(\alpha,\beta)}$ with $\alpha>-1$, \cite{szegő1975orthogonal}*{Theorem 8.21.12} gives a Bessel-type approximation near $\theta=0$.
Specializing to $\alpha=1$, $\beta=0$, and $n=L$, one obtains uniformly for $0<\theta\le \theta_0<\pi$:
\begin{equation*}
    \sin(\theta/2)\,P_L^{(1,0)}(\cos\theta)
    =
    \left(\frac{\theta}{\sin\theta}\right)^{1/2}
    J_1\bigl((L+1)\theta\bigr)
    +\mathcal{O}(L^{-1}).
\end{equation*}
Therefore,
\begin{equation}\label{eq:07}
    P_L^{(1,0)}(\cos\theta)
    =
    \left(\frac{\theta}{\sin\theta}\right)^{1/2}
    \frac{J_1((L+1)\theta)}{\sin(\theta/2)}
    +\mathcal{O}(L^{-1}).
\end{equation}
Split the integral $I_{L}$ as
\begin{multline*}
    I_L
    =
    \frac{1}{4(L+1)^2}
    \left(\int_{0}^{\theta_0}+\int_{\theta_0}^{\pi}\right)
    \ell(\theta)\,(P_L^{(1,0)}(\cos\theta))^2\,\sin\theta\,d\theta
    \\
    =:    \frac{1}{4(L+1)^2}
    \left(I_L^{(0)}+I_L^{(\mathrm{bulk})}\right),
\end{multline*}
where $\theta_0\in(0,\pi)$ is fixed.
By the Darboux-Szeg\H{o} asymptotic formula for Jacobi polynomials (see Szeg\H{o} \cite{szegő1975orthogonal}*{Theorems~8.21.8-9}), for every fixed $\theta_0>0$ one has
\begin{equation*}
    P_L^{(1,0)}(\cos\theta)
    =
    \frac{1}{\sqrt{\pi L}}
    \frac{\cos\!\big((L+1)\theta-\tfrac{3\pi}{4}\big)}{(\sin(\theta/2))^{3/2}(\cos(\theta/2))^{1/2}}
    +\mathcal{O}(L^{-3/2}),
\end{equation*}
uniformly for $\theta\in[\theta_0,\pi]$.
In particular,
\begin{equation*}
    \bigl(P_L^{(1,0)}(\cos\theta)\bigr)^2
    =
    \mathcal{O}(L^{-1}) \quad\text{uniformly on }[\theta_0,\pi].
\end{equation*}
Since $\ell(\theta)\sin\theta$ is bounded on $[\theta_0,\pi]$, it follows that
\begin{equation*}
    I_L^{(\mathrm{bulk})}
    =
    \int_{\theta_0}^{\pi}
    \ell(\theta)\,\bigl(P_L^{(1,0)}(\cos\theta)\bigr)^2\,\sin\theta\,d\theta
    = 
    \mathcal{O}(L^{-1}).
\end{equation*}
For the neighborhood of $0$, we insert the approximation \eqref{eq:07} into \(I_L^{(0)}\) along with the expansion $\sin(\theta/2)^2=\frac{\theta^2}{4} (1+\mathcal{O}(\theta^2))$ as $\theta\to 0$.
This gives
\begin{equation*}
    I_L^{(0)}
    =
    \int_{0}^{\theta_0}
    4
    \ell(\theta)\,
    \frac{J_1((L+1)\theta)^2}{\theta}\,d\theta
    +\mathcal{O}(L^{-1}).
\end{equation*}
Making the change of variables $x=(L+1)\theta$, we get
\begin{equation*}
    I_L^{(0)}
    =
    \int_{0}^{(L+1)\theta_0}
    4
    \ell\left(\frac{x}{L+1}\right)
    \frac{J_1(x)^2}{x}\,dx
    +\mathcal{O}(L^{-1}).
\end{equation*}
As \(\theta\to0\), we have
\begin{equation*}
\ell(\theta)=\frac{\theta}{2}+\mathcal{O}(\theta^2).
\end{equation*}
Therefore,
\begin{equation*}
    \ell\left(\frac{x}{L+1}\right)
    =
    \frac{x}{2(L+1)}
    +\mathcal{O}\left(\frac{x^2}{(L+1)^2}\right).
\end{equation*}
Substituting this into the integral yields
\begin{equation}\label{eq:09}
    I_L^{(0)}
    =
    \frac{2}{(L+1)}
    \int_{0}^{(L+1)\theta_0} J_1(x)^2\,dx
    +\mathcal{O}(L^{-1}).
\end{equation}

\subsubsection{Growth of the Bessel integral}

We recall the classical asymptotic expansion of the Bessel function $J_\nu$, valid for fixed $\nu$ and $x\to\infty$ (see, e.g., \cite{watson1995treatise}*{Chapter VII}, or \cite{DLMF}*{S. 10.17(ii)}):
\begin{equation*}
J_\nu(x)
=
\sqrt{\frac{2}{\pi x}}
\left(
\cos \left(x-\frac{\nu\pi}{2}-\frac{\pi}{4}\right)
+ \mathcal{O}(x^{-1})
\right).
\end{equation*}
For $\nu=1$, this gives
\begin{equation*}
J_1(x)
=
\sqrt{\frac{2}{\pi x}}
\left(
\cos \left(x-\frac{3\pi}{4}\right)
+ \mathcal{O}(x^{-1})
\right),
\qquad x\to\infty.
\end{equation*}
Squaring, we obtain
\begin{align*}
J_1(x)^2
&=
\frac{2}{\pi x}
\cos^2 \left(x-\frac{3\pi}{4}\right)
+ \mathcal{O}(x^{-2}) \\
&=
\frac{1}{\pi x}
\Bigl(1+\cos(2x-\tfrac{3\pi}{2})\Bigr)
+ \mathcal{O}(x^{-2}),
\qquad x\to\infty,
\end{align*}
where we used the identity $\cos^2 y=\frac12(1+\cos 2y)$.

Although $J_1(x)^2$ contains an oscillatory term of order $x^{-1}$, its integral exhibits logarithmic growth. 
Indeed, integrating term by term and using that
\begin{equation*}
\int_{1}^\infty \cos(2x-\tfrac{3\pi}{2})\,\frac{dx}{x}
\end{equation*}
converges conditionally (see again \cite{watson1995treatise}*{Chapter VII}), we obtain
\begin{equation}\label{eq:10}
\int_0^A J_1(x)^2\,dx
=
\frac{1}{\pi}\int_1^A \frac{dx}{x} + \mathcal{O}(1)
=
\frac{1}{\pi}\log A + \mathcal{O}(1),
\qquad A\to\infty.
\end{equation}
Taking $A=(L+1)\theta_0$ and applying \eqref{eq:10} in \eqref{eq:09} we get
\begin{equation*}
    I_L^{(0)}
    =
    \frac{2 \log (L+1)}{\pi (L+1)}
    +\mathcal{O}(L^{-1}).
\end{equation*}
Combining this with the bulk estimate, we obtain
\begin{equation*}
    I_L
    =
    \frac{\log (L+1)}{2\pi (L+1)^3}
    +
    \mathcal{O}(L^{-3}),
    \qquad L\to\infty.
\end{equation*}
%

\subsubsection{Proof of \cref{prop:02}}

Recall that $r=(L+1)^2$ and $N=rk$.
By \cref{prop:dpp:iso} and the preceding results, we conclude:
\begin{equation*}
    \Esp_{\omega_N \sim \HopfHarmonicensemble{k}{r}}
    \left[
    \Energy_{0}(\omega_N)
    \right]
    =
    -
    rk\log k
    -
    \frac{k^2 r^2}{4}
    +
    \frac{k^2 \sqrt{r} \log (r)}{4\pi }
    +
     k^2 
     \mathcal{O}\left( \sqrt{r} \right)
     .
\end{equation*}

\subsubsection{Proof of \cref{cor:Harm:asymp}}

We take $k=\left\lfloor\frac{\sqrt{r}}{\log r}\right\rfloor$ and let $ r\to\infty$.
Then
\begin{equation*}
k
=
\frac{\sqrt r}{\log r}
\left(1+o(1)\right),
\qquad
N=kr
=
\frac{r^{3/2}}{\log r}
\left(1+o(1)\right).
\end{equation*}
Inverting this relation gives
\begin{equation*}
r
=
N^{2/3}\log(N)^{2/3}
\left(1+o(1)\right),
\qquad
k
=
\frac{N^{1/3}}{\log(N)^{2/3}}
\left(1+o(1)\right).
\end{equation*}
We now expand the different terms of the expression in \cref{prop:02}.
First,
\begin{equation*}
\log k
=
\frac13\log N
-
\frac23\log\log N
+
o(1),
\end{equation*}
and therefore
\begin{equation*}
-N\log k
=
-\frac13 N\log N
+
\frac23 N\log\log N
+
o(N).
\end{equation*}
Next,
\begin{equation*}
k^2 \sqrt r \log r
=
\frac{r}{\log(r)^2}\,\sqrt r \log r
\left(1+o(1)\right)
=
\frac{r^{3/2}}{\log r}
\left(1+o(1)\right)
=
N\left(1+o(1)\right),
\end{equation*}
so that
\begin{equation*}
\frac{k^2 \sqrt r \log r}{4\pi}
=
\mathcal O(N).
\end{equation*}
Similarly,
\begin{equation*}
k^2 \mathcal O(\sqrt r)
=
\mathcal O\!\left(\frac{r^{3/2}}{\log(r)^2}\right)
=
o(N).
\end{equation*}
Collecting all contributions yields
\begin{equation*}
\Esp_{\omega_N \sim \HopfHarmonicensemble{k}{r}}
\!\left[
\Energy_{0}(\omega_N)
\right]
=
-\frac{N^2}{4}
-
\frac13 N\log N
+
\frac23 N\log\log N
+
\mathcal O(N),
\end{equation*}
which concludes the proof.


\appendix
\section{The Diamond ensemble}
The Diamond ensemble is a family of points defined in \cites{BeltranEtayo2020,BeltranEtayoLopezGomez2023} with the objective of minimizing the discrete logarithmic energy on $\S^2$. 
This quasi-deterministic family actually allows us to compute analytically the asymptotic expansion of the energy, providing a rigorous proof beyond numerical results. 

The Diamond ensemble is defined from unity roots on parallels determined by given latitudes. Thus, it is defined by the number of parallels, $p$, their latitudes, $z_j$, and the number of roots per parallel, $r_j$.
However, points in this set are random, as they are defined based on $\theta_j$, uniformly distributed random variables in the interval $[0, 2\pi)$:
\begin{multline}\label{ptos_diamante}
    \Omega (p, \{r_j\}, \{z_j\}) = \{\mathbf{x}_j^i\} \\
    =\left\{\left(\sqrt{1-z_j^2}\cos\left(\frac{2\pi i }{r_j} + \theta_j\right), \sqrt{1-z_j^2}\sin\left(\frac{2\pi i }{r_j} + \theta_j\right), z_j\right)\right\}.
\end{multline}
The expectation of the logarithmic energy of the points described above is then given by
\begin{multline*}
    \Esp_{\theta_1,...,\theta_p}[\Energy_{0}(\Omega (p, \{r_j\}, \{z_j\})] =  -2\log(2) \\
    - \sum_{j=1}^p r_j\left[ \log(4) + \frac{1}{2}\log(1-z_j^2)+\log r_j\right] - \sum_{j,k=1}^p\frac{\log(1-z_j z_k+|z_j-z_k|)}{2},
\end{multline*}
see \cite{BeltranEtayo2020}*{Proposition 2.4}.
The previous formula provides us with the value of $z_l$ that minimizes the logarithmic energy given the different $\{r_j\}$, which is given by:
\begin{equation}\label{z_Diamond}
    z_l=\frac{\displaystyle\sum_{j=l+1}^p r_j - \displaystyle\sum_{j=1}^{l-1} r_j}{1 + \displaystyle\sum_{j=1}^p r_j} = 1- \frac{1 + r_l + 2\displaystyle\sum_{j=1}^{l-1} r_j}{N-1},
\end{equation}
where $N=2 +\sum_{j=1}^p r_j$ is the total number of points.

Thus, the remaining task is to choose the number of roots of unity $r_l$ on each parallel $l$ ( for a given number of parallels $p$).
If we temporary allow $r_l$ to be non-integer, a natural assumption is that the distance between points along each parallel should be proportional to the distance between parallels.
In other words, one might set
$$ r_j=\frac{K_0\pi\sin\left(\frac{j\pi}{p+1}\right)}{\sin\left(\frac{\pi}{2(p+1)}\right)}
$$
for some constant $K_0$.
Using this continuous ansatz and expanding the energy, one finds the optimal constant to be $K_0=3/\pi$. 

However, this construction is not possible since the number of points in each parallel must be a natural number. 
Beltrán and Etayo \cite{BeltranEtayo2020} introduced a piecewise-linear approximation to ${r_j}$, and a further improved (not necessarily continuous) selection was given in \cite{BeltranEtayoLopezGomez2023}.
For the optimal selection of ${r_j}$, the associated energy reads
\begin{equation*}
\Esp_{\theta_1,...,\theta_p}\left[\Energy_{0}(\diamond_N)\right]
=
\left( \frac{1}{2} - \log 2 \right)N^2-\frac{1}{2}N\log N+c_{\diamond}N+o(N)    
\end{equation*}
where $c_{\diamond}=-0.049222...$ that is quite close to the lower bound of the minimal logarithmic energy, see \cref{minlog2}.


\subsection{Fibres of the Diamond ensemble}\label{calculo_Diamond}

The preimage of the Hopf fibration, according to \cref{inversa_hopf}, is not defined at the point $(-1,0,0)$ 
To avoid this point, we consider the Diamond ensemble without the poles and then rotate the set to place the South pole at $(-1,0,0)$. 
With this rotation applied to  \cref{ptos_diamante}, the points of the Diamond ensemble take the form
$$\mathbf{x}_j^i=\left( z_j, \sqrt{1-z_j^2}\cos\left(\frac{2\pi i}{r_j}+\theta_j\right),\sqrt{1-z_j^2}\sin\left(\frac{2\pi i}{r_j}+\theta_j\right) \right),$$
whose fibre by the inverse of the Hopf fibration is, after simplification,
\begin{multline*}
    h^{-1}(\mathbf{x}_j^i)
    =\frac{1}{\sqrt{2}}
    \Bigl(\sqrt{1+z_j}\cos t, \sqrt{1+z_j}\sin t, \Big.
    \\
    \left.\sqrt{1-z_j}\sin \left(t-\frac{2\pi i}{r_j}-\theta_j\right), \sqrt{1-z_j}\cos \left(t-\frac{2\pi i}{r_j}-\theta_j\right)\right).
\end{multline*}
We take roots of unity along each fibre, with a uniformly distributed random phase, and denote the corresponding rotation angle by $\psi$.
We have
$$t=\frac{2\pi l}{k}+\psi_j^i \;\text{ and }\; \phi=\frac{2\pi i}{r_j} + \theta_j,$$
with $1\leq l \leq k$, where $k$ is a positive integer to be determined.

When calculating the distance between two points, we have
\begin{multline*}
    ||\mathbf{x}_{j_1}^{i_1l_1} - \mathbf{x}_{j_2}^{ i_2l_2}|| =\left[2 - \sqrt{1 + z_{j_1}}\sqrt{1 + z_{j_2}}\cos\left(\frac{2\pi}{k}(l_1-l_2) + \psi_{j_1}^{i_1} - \psi_{j_2}^{i_2}\right)  \right.
    \\
    \left.-\sqrt{1 - z_{j_1}}\sqrt{1 - z_{j_2}}\cos\left(\frac{2\pi}{k}(l_1-l_2) - \frac{2\pi i_1}{r_{j_1}}+ \frac{2\pi i_2}{r_{j_2}} + \psi_{j_1}^{i_1} - \psi_{j_2}^{i_2} - \theta_{j_1}+\theta_{j_2}\right)\right]^{1/2}.
\end{multline*}
So, the calculation of the energy corresponds to the sum of two quantities:
\begin{enumerate}
    \item The sum of the logarithms of the inverse of the distance between all different pairs of points that are in the same fibre, denoted as quantity $A$.
    By \cref{prop_energia_fibra}, $A =-N\log k$.
    \item The expected sum of the logarithms of the inverse of the distance between all different pairs of points belonging to different fibres, denoted as quantity $B$.
\end{enumerate}

\subsubsection{Quantity \textit{B}}

To simplify the notation, we denote the distance between any two points belonging to different fibres as follows:
\begin{multline*}
||\mathbf{x}_i-\mathbf{x}_j||=\left(2 - \sqrt{1 + z_i}\sqrt{1 + z_j}\cos(\psi_i-\psi_j) - 
\right.\\\left.
\sqrt{1 - z_i}\sqrt{1 - z_j}\cos((\psi_i-\psi_j)-(\theta_i-\theta_j))\right)^{1/2},
\end{multline*}
where $\psi$ indicates the angle in the fibre and $\theta$ the angle in the parallel of $\S^2$.
Given two points from different fibres, computing the expectation of the logarithm of their distance is
\begin{multline*}
    \Esp_{\theta_i, \theta_j, \psi_i, \psi_j}\left[-\log||\mathbf{x}_i-\mathbf{x}_j||\right]=\frac{1}{(2\pi)^4}\int_0^{2\pi}\int_0^{2\pi}\int_0^{2\pi}\int_0^{2\pi}
    \\
    -\scalebox{0.8}{$\log\left(\sqrt{2 - \sqrt{1 + z_i}\sqrt{1 + z_j}\cos(\psi_i-\psi_j) - \sqrt{1 - z_i}\sqrt{1 - z_j}\cos((\psi_i-\psi_j)-(\theta_i-\theta_j))}\right)d\theta_i d\theta_j d\psi_i d\psi_j$}
    \\
    =\scalebox{0.92}{$\frac{-1}{2^3\pi^2}\int_0^{2\pi}\int_0^{2\pi}\log\left(2 - \sqrt{1 + z_i}\sqrt{1 + z_j}\cos(\psi) - \sqrt{1 - z_i}\sqrt{1 - z_j}\cos(\psi-\theta)\right)d\theta d\psi.$}
\end{multline*}
By applying \cref{lema_ruso} we obtain
\begin{multline}\label{eq:03}
    \Esp_{\theta_i, \theta_j, \psi_i, \psi_j}\left[-\log||\mathbf{x}_i-\mathbf{x}_j||\right]
    \\
    =
    \scalebox{1}{$\frac{-1}{4\pi}\int_0^{2\pi}\log\left(\frac{2-\sqrt{1 + z_i}\sqrt{1 + z_j}\cos(\psi)+\sqrt{(2-\sqrt{1 + z_i}\sqrt{1 + z_j}\cos(\psi))^2-(\sqrt{1 - z_i}\sqrt{1 - z_j})^2}}{2}\right)d\psi.$}
\end{multline}
Unfortunately, we have not been able to solve the last elliptic integral analytically. 
This has been attempted through manual methods, computational methods, and by consulting tables of integrals, all with negative results. 
As a result, this approach for the analytical calculation of the asymptotic expansion has been abandoned, and a numerical approach has been adopted instead, as illustrated in \cref{fig:primer_orden,fig:sec_orden}.



\bibliographystyle{amsplain}
\bibliography{Bibliography}


\end{document}